\documentclass[11pt,a4paper,reqno]{amsart} 
\usepackage{amssymb}
\usepackage{latexsym}
\usepackage{amsmath}
\usepackage{mathrsfs}
\usepackage{upgreek}
\usepackage{mathabx}
\usepackage[X2,T1]{fontenc}
\usepackage{cite}
\usepackage{color}
\usepackage{calc}                   

\newcommand{\loc}{\operatorname{loc}}
\newcommand{\scal}[2]{\langle #1,#2\rangle}
\newcommand{\rr}[1]{\mathbf R^{#1}}

\newcommand{\nn}[1]{\mathbf N^{#1}}
\newcommand{\zz}[1]{\mathbf Z^{#1}}
\newcommand{\cc}[1]{\mathbf C^{#1}}
\newcommand{\nm}[2]{\Vert #1\Vert _{#2}}

\newcommand{\op}{\operatorname{Op}}

\newcommand{\sets}[2]{\{ \, #1\, ;\, #2\, \} }

\newcommand{\cdo}{\, \cdot \, }

\newcommand{\eabs}[1]{\langle #1\rangle}     

\newcommand{\vrum}{\vspace{0.1cm}}

\newcommand{\sumvec}{\operatorname{sum}}

\newcommand{\maclB}{\mathcal B}

\newcommand{\maclS}{\mathcal S}

\newcommand{\mascF}{\mathscr F}
\newcommand{\mascH}{\mathscr H}

\newcommand{\mascP}{\mathscr P}
\newcommand{\mascS}{\mathscr S}

\newcommand{\RE}{\operatorname{Re}}
\newcommand{\IM}{\operatorname{Im}}

\newcommand{\maclA}{\mathcal A}

\newcommand{\maclH}{\mathcal H}

\newcommand{\indlim}{\operatorname{ind \, lim\, }}
\renewcommand{\projlim}{\operatorname{proj \, lim\, }}

\numberwithin{equation}{section}          

\newtheorem{thm}{Theorem}
\numberwithin{thm}{section}
\newtheorem*{tom}{\rubrik}
\newcommand{\rubrik}{}
\newtheorem{prop}[thm]{Proposition}
\newtheorem{cor}[thm]{Corollary}
\newtheorem{lemma}[thm]{Lemma}

\theoremstyle{definition}

\newtheorem{defn}[thm]{Definition}

\theoremstyle{remark}
\newtheorem{rem}[thm]{Remark}              

\author{Joachim Toft}

\address{Department of Mathematics,
Linn{\ae}us University, V{\"a}xj{\"o}, Sweden}

\email{joachim.toft@lnu.se}

\author{Divyang G. Bhimani}

\address{Department of Mathematics,
Indian Institute of Science Education and Research, Pune, India}

\email{divyang.bhimani@iiserpune.ac.in}

\author{Ramesh Manna}

\address{School of Mathematical Sciences, National Institute of
Science Education and Research, Bhubaneswar, An OCC of
Homi Bhabha National Institute, Jatni 752050, India.}

\email{rameshmanna@niser.ac.in}

\title{Fractional Fourier transforms, harmonic oscillator
propagators and Strichartz estimates
on Pilipovi{\'c} and modulation spaces}



\begin{document}

\begin{abstract}
We give a proof of that harmonic oscillator
propagators and fractional
Fourier transforms are essentially the same.
We deduce continuity properties and fix time
estimates for such operators on  modulation
spaces, and apply the results to prove
Strichartz estimates for such propagators
when acting on Pilipovi{\'c} and modulation
spaces. Especially we extend some results by
Balhara, Cordero, Nicola, Rodino and Thangavelu. 
We also show that general forms of fractional 
harmonic oscillator propagators 
are continuous on suitable Pilipovi{\'c} spaces.
\end{abstract}

\maketitle

\section{Introduction}\label{sec0} 

%
%
%
%
%
%
%
%

%
%

\par

In the paper we investigate mapping properties for
powers of harmonic oscillators, their
propagators and fractional Fourier
transforms on Pilipovi{\'c} spaces and modulation spaces. 
Especially we link fractional Fourier transforms with
harmonic oscillator propagators, which opens up for full 
transitions of various properties between these operators.

\par

We also deduce certain continuity properties for
fractional Fourier transforms on modulation spaces (including
Wiener amalgam spaces). By using the link between the fractional
Fourier transform and harmonic oscillator propagators we
extend at the same time certain continuity properties
of harmonic oscillator propagators on modulation spaces,
given in \cite{Bhi,BhiBalTha,CorGroNicRod,CorNic}.
These investigations are also related to \cite{ChFuGrWu},
where among others, $L^p$ and Hausdorff-Young estimates
for fractional Fourier transforms are established. Thereafter
we apply such continuity results to extend certain Strichartz
estimates for harmonic oscillator propagators in \cite{CorNic}
when acting on modulation spaces.
In the end we apply
our results to certain general classes of
time-dependent equations, similar to
Schr{\"o}dinger and heat equations. We prove that
several of these equations are ill-posed
in the framework of Schwartz space,
Gelfand-Shilov spaces and their distribution spaces,
but well-posed in the framework of suitable
Pilipovi{\'c} spaces.

\par

Harmonic oscillators and their propagators are important in
quantum mechanics, e.{\,}g. when investigating free particles
in quantum systems. An important question concerns
continuity for
such operators. For the (standard) harmonic oscillator
$$
H_x=|x|^2-\Delta _x,
\qquad
x\in \rr d,
$$
the corresponding propagator is given by
\begin{equation}\label{Eq:IntroHOscProp}
P_{\varrho} = e^{-i\varrho H_x},
\end{equation}
which can also be formulated as
\begin{equation}\tag*{(\ref{Eq:IntroHOscProp})$'$}
P_{\varrho} = e^{-iH_{x,\varrho}},
\qquad
H_{x,\varrho} = \varrho (|x|^2-\Delta _x).
\end{equation}
Here $\varrho \in \mathbf R$.
(See Sections \ref{sec1}--\ref{sec4} for more general
operators of such forms.)
It is well-known that these operators
are homeomorphisms
on the Schwartz space $\mascS (\rr d)$
and its dual $\mascS '(\rr d)$, the set of tempered 
distributions.
(See \cite{Ho1} and Section \ref{sec1} for notations.)
Similar continuity properties hold true with Pilipovi{\'c}
spaces $\maclH _{0,s}(\rr d)$ and $\maclH _s(\rr d)$,
and their distribution spaces
$\maclH _{0,s}'(\rr d)$ and $\maclH _s'(\rr d)$,
in place of $\mascS (\rr d)$ and $\mascS '(\rr d)$.
We recall that Fourier invariant (standard)
Gelfand-Shilov are special cases of Pilipovi{\'c}
spaces. More precisely we have
\begin{alignat*}{4}
\maclH _s(\rr d) &= \maclS _s(\rr d) \neq \{ 0 \} ,&
\quad s&\ge \frac 12, &
\qquad
\maclH _s(\rr d) &\neq \maclS _s(\rr d) = \{ 0 \} ,&
\quad s&< \frac 12,
\\[1ex]
\maclH _{0,s}(\rr d) &= \Sigma _s(\rr d) \neq \{ 0 \} ,&
\quad s&> \frac 12, &
\qquad
\maclH _{0,s}(\rr d) &\neq \Sigma _s(\rr d) = \{ 0 \} ,&
\quad s&\le \frac 12
\end{alignat*}
(see \cite{ReSi,Toft18}).

\par

Harmonic oscillators and their propagators also possess 
convenient
mapping properties in the background of suitable modulation 
spaces,
a family of (quasi-)Banach spaces which were introduced
in \cite{Fei1} by H. Feichtinger and further developed in 
\cite{Fei6,FeiGro1,FeiGro2,GaSa,Gc1}. For example, in
\cite{KatKobIto1,KatKobIto2,KatKobIto3,KatKobIto4},
several continuity results for Schr{\"o}dinger 
propagators including potential terms
when acting on modulation spaces are deduced. Harmonic
oscillator propagators are then obtained by choosing 
the potentials as
$c|x|^2$ for some positive constant $c$. For example, 
it is proved in 
\cite{KatKobIto1,KatKobIto2,KatKobIto3}
that for the Fourier invariant modulation space 
$M^p(\rr d)$, the map
\begin{equation}\label{Eq:FirstHarmPropContMod}
e^{iH_{x,\varrho}}\, :\, M^p(\rr d)\to M^p(\rr d)
\end{equation}
is continuous. See also 
\cite{Bhi,BhiBalTha,CorGroNicRod,CorNic}
for other results concerning mapping properties of
propagators on modulation spaces.

\par

In Section \ref{sec3} we give a strict proof, based on
the Bargmann transform,
of that the propagator in
\eqref{Eq:IntroHOscProp} can be identified with
fractional Fourier transforms by the formula
\begin{equation}\label{Eq:IntroFracFourHarmOscProp}
e^{-i\frac {\pi} 4H_{x,\varrho}}
=
e^{-i\frac {\pi \varrho  d}4}
\mascF _{\! \varrho} ,
\end{equation}
for every $\varrho \in \mathbf R$.
Here recall that the fractional Fourier transform
$\mascF _{\! \varrho}$, acting on $\mascS (\rr d)$
is given by
\begin{equation*}
\mascF _{\! \varrho} f(\xi )
=
\scal {K_{d,\varrho} (\xi ,\cdo )}f,
\end{equation*}
where $K_{d,\varrho} (\xi ,x)$ is the distribution
kernel, given by
\begin{equation}
K_{d,\varrho}
=
\bigotimes _{j=1}^d K_\varrho ,
\label{Eq:FracFourKernel}
\end{equation}
with
\begin{equation}
K_\varrho (\xi ,x)
=
\begin{cases}
\left (
\frac {1-i\cot (\frac {\pi \varrho }2)}{2\pi}
\right )^{\frac 12}
\exp
\left ( i\cdot
\frac {(x^2+\xi ^2)\cos ({\frac {\pi \varrho }2})
-2{\xi x}}{2\sin (\frac {\pi \varrho}2)}
\right ), & \! \varrho \in \mathbf R\setminus 2\mathbf Z,
\\[1ex]
\delta (\xi -x), & \! \varrho \in 4\mathbf Z,
\\[1ex]
\delta (\xi +x), & \! \varrho \in 2+4\mathbf Z,
\end{cases}
\label{Eq:FracFourKernelDim1}
\end{equation}
Here $x,\xi \in \mathbf R$.
(See e.{\,}g. \cite{Alm,OzaZalKut} and the references
therein. See also Section \ref{sec1} for more details on
fractional Fourier transform.) We observe that
\begin{align*}
(\mascF _1f)(\xi ) &= \frac 1{(2\pi )^{\frac d2}}\int _{\rr d}
f(x)e^{-i\scal x\xi}\, dx
\intertext{and}
(\mascF _3f)(\xi ) &= \frac 1{(2\pi )^{\frac d2}}\int _{\rr d}
f(x)e^{i\scal x\xi}\, dx
\end{align*}
are the (ordinary) Fourier transform and the inverse Fourier
transform, respectively.

\par

Ideas on fractional Fourier transforms goes back to
at least 1929 (cf. historical notes and references in
\cite{BulSul}). On the other hand, the first rigorous
approach seems to be given at first 1939 by Kober in
\cite{Kob}. Since then numerous applications of
fractional Fourier transforms have appeared.
For example, they were explicitly introduced in
quantum mechanics around 1980 and thereafter applied in
optics (see e.{\,}g. \cite{FanHu,Nam} and the references therein).
In quantum mechanics,
the formulae \eqref{Eq:FracFourKernel} and
\eqref{Eq:FracFourKernelDim1} appear naturally by considering
certain rotations in the phase space and their induced actions
on quantum observables (see e.{\,}g.
\cite{CorGroNicRod,CorNic,dGo1,dGo2,dGoLue}
and Remark \ref{Rem:FracFour} in Section \ref{sec1}).
There are also several 
applications in signal analysis e.{\,}g. when discussing
rotation properties for time-frequency representations
in time-frequency analysis, phase retrieval,
optics and pattern recognition.
(See e.{\,}g. \cite{AliBas,Alm,FanHu,OzaZalKut}
and the references therein.)
Here we also remark that in some aspects, the theory of
the fractional Fourier transform is merely a special case
of the metaplectic representation (see \cite{dGoLue}).

\par

In terms of the Bargmann
transform $\mathfrak V_d$, fractional Fourier transforms
take the convenient form
\begin{equation}\label{Eq:BargApprFracFourTrans}
(\mathfrak V_d(\mascF _{\! \varrho} f))(z)
=
(\mathfrak V_df)(e^{-i\frac {\pi \varrho}2}z),
\qquad
z\in \cc d.
\end{equation}
(See e.{\,}g. \cite{Ba1,Toft18}.)
The relation \eqref{Eq:BargApprFracFourTrans}
was obtained for $\varrho =1$ already in
\cite{Ba1} by Bargmann.

\par

A motivation of the identity
\eqref{Eq:IntroFracFourHarmOscProp} is
given in e.{\,}g. \cite{KutOza},
where several links between
the harmonic oscillator and fractional Fourier
transforms are established. We also remark that
F. G. Mehler established a formula
(afterwards named \emph{Mehler's formula})
for the operator $e^{-H_{x,\varrho}}$, $\varrho >0$,
already in \cite{Meh}. Observe that
$e^{-H_{x,\varrho}}$ is the canonical density operator
in statistical physics, see \cite[Section 3.4]{SakNap}.

\par

Analytic extensions of Mehler's formula lead to
that the kernel of $e^{-i\frac \pi 4H_{x,\varrho}}$
is essentially the same as the kernel of the fractional
Fourier transform \eqref{Eq:FracFourKernel} and
\eqref{Eq:FracFourKernelDim1}.

\par

In the aftermath of \eqref{Eq:IntroFracFourHarmOscProp}
one may link several results in e.{\,}g.
\cite{Toft10,Toft18}
concerning mapping properties of
fractional Fourier transforms on modulation spaces
with analogous results in e.{\,}g.
\cite{Bhi,BhiBalTha,CorGroNicRod,CorNic,KatKobIto1,
KatKobIto2,KatKobIto3,KatKobIto4}
concerning mapping properties of the propagators
on modulation spaces. Some explicit demonstration of
such transfers are given in Section \ref{sec3}.
Here it is addressed
that $e^{iH_{x,\varrho}}$ is homeomorphic on
$M^q_{(\omega )}(\rr d)$
for suitable weight $\omega$, because the
fractional Fourier transforms possess the same
continuity properties in view of
\cite[Proposition 7.1]{Toft18}.
(See Propositions \ref{Prop:FracFourMod} and
\ref{Prop:HarmPropMod} in Section \ref{sec3}.)
More generally, in Section \ref{sec3} we show
that for $\varrho \in \mathbf R\setminus 2\mathbf Z$
and $q\le p$,
then $\mascF _{\! \varrho}$ is continuous
from $M^{p,q}_{(\omega )}(\rr d)$ to 
$M^{q,p}_{(\omega )}(\rr d)$,
and from $W^{q,p}_{(\omega )}(\rr d)$ to 
$W^{p,q}_{(\omega )}(\rr d)$,
for suitable weights $\omega$.
Again, by using the link
\eqref{Eq:IntroFracFourHarmOscProp} we transfer 
these continuity
properties for fractional Fourier
transforms to harmonic
oscillator propagators. In fact, the following proposition
is obtained by letting the weights in
Theorems \ref{Thm:FrFTModSpec1} and \ref{Thm:FrFTModSpec2}
in Section \ref{sec3}, be trivially equal to one.
(See also Proposition
\ref{Prop:ConseqMainThms}$'$ in Section \ref{sec3}.)

\par

\begin{prop}\label{Prop:ConseqMainThms}
Let $\varrho \in \mathbf R$
and $p,q\in (0,\infty]$ be such that $q\le p$.
Then the following is true:
\begin{enumerate}
\item the map
\begin{alignat*}{2}
\mascF _{\! \varrho} =e^{-i\frac \pi 4H_{x,\varrho}}
\, &:\, M^{p,q}(\rr d)+W^{q,p}(\rr d) &
&\to
M^{q,p}(\rr d)+W^{p,q}(\rr d)
\end{alignat*}
is continuous;

\vrum

\item if in addition $\varrho \notin \mathbf Z$, then the map
\begin{alignat*}{2}
\mascF _{\! \varrho} =e^{-i\frac \pi 4H_{x,\varrho}}
\, &:\, M^{p,q}(\rr d)+W^{q,p}(\rr d) &
&\to
M^{q,p}(\rr d){\textstyle \bigcap}W^{p,q}(\rr d)
\end{alignat*}
is continuous.
\end{enumerate}
\end{prop}

\par

There are several types of estimates behind
the conclusions in the previous proposition,
which are expected to be applicable in
non-linear partial differential equations.
Examples on such estimates are
\begin{alignat*}{2}
\nm {\mascF _{\! \varrho} f}{M^{q,p}}
&=
\nm {e^{-i\frac \pi 4H_{x,\varrho}}f}{M^{q,p}}
\lesssim
|\sin (\textstyle{\frac {\pi \varrho}2})|^{d(\frac 1p-\frac 1q)}
\nm f{M^{p,q}},
\\[1ex]
\nm {\mascF _{\! \varrho} f}{M^{q,p}}
&=
\nm {e^{-i\frac \pi 4H_{x,\varrho}}f}{M^{q,p}}
\lesssim
|\cos (\textstyle{\frac {\pi \varrho }2})|
^{d(\frac 1p-\frac 1q)}
\nm f{W^{q,p}},
\\[1ex]
\nm {\mascF _{\! \varrho} f}{W^{p,q}}
&=
\nm {e^{-i\frac \pi 4H_{x,\varrho}}f}{W^{p,q}}
\lesssim
|\cos (\textstyle{\frac {\pi \varrho}2})|
^{d(\frac 1p-\frac 1q)}
\nm f{M^{p,q}},
\intertext{and}
\nm {\mascF _{\! \varrho} f}{W^{p,q}}
&=
\nm {e^{-i\frac \pi 4H_{x,\varrho}}f}{W^{p,q}}
\lesssim
|\sin (\textstyle{\frac {\pi \varrho}2})|
^{d(\frac 1p-\frac 1q)}
\nm f{W^{q,p}},
\end{alignat*}
still with $q\le p$ (see Theorems \ref{Thm:FrFTModSpec1}
and \ref{Thm:FrFTModSpec2}).

\medspace

Some of our investigations include more 
general propagators and fractional Fourier transforms,
where $\varrho$ in 
\eqref{Eq:IntroHOscProp}$'$ and
\eqref{Eq:IntroFracFourHarmOscProp} is allowed to
be any complex number.
If $\IM (\zeta )>0$, then $P_{\varrho}$
does neither make sense as a continuous
operator on $\mascS (\rr d)$,
nor on any Fourier invariant Gelfand-Shilov space
and their duals.
Consequently, in order to investigate such extended
class of propagators, or, even more generally,
powers of $H_{x,\varrho}$ and their propagators, i.{\,}e. operators
of the forms
\begin{equation}\label{Eq:IntroFracPowHOscProp}
H_{x,\varrho}^r
\quad \text{and}\quad
P_{\varrho, r} = e^{-iH_{x,\varrho}^r},
\qquad \varrho \in \mathbf C,\ r\in \mathbf R,
\end{equation}
other families of function and distribution spaces
are needed.
It turns out that such continuity discussions can
be performed in the framework of
certain Pilipovi{\'c} spaces and their distribution
spaces.

\par

In order to shed some further lights,
we present the following propositions,
which are immediate consequences of our
investigations. For the fractional Fourier transforms,
these conclusions also follows from the analysis
in \cite{Toft18}.

\par

\begin{prop}\label{Prop:IntroFracFourComplOrder}
Let $\varrho \in \mathbf C$ and $s\in \overline{\mathbf R}_\flat$.
Then the following is true:
\begin{enumerate}
\item if $s< \frac 12$, then $\mascF _{\! \varrho}$
and $e^{-i\frac \pi 4H_{x,\varrho}}$ are homeomorphisms on
$\maclH _s(\rr d)$ and on $\maclH _s'(\rr d)$;

\vrum

\item if $\IM (\varrho )<0$ and $s\ge \frac 12$, then
$\mascF _{\! \varrho}$
and $e^{-i\frac \pi 4H_{x,\varrho}}$ are continuous
injections but not surjections on
$\maclH _s(\rr d)$, $\mascS (\rr d)$, $\mascS '(\rr d)$
and on $\maclH _s'(\rr d)$;

\vrum

\item if $\IM (\varrho )=0$, then
$\mascF _{\! \varrho}$
and $e^{-i\frac \pi 4H_{x,\varrho}}$ are homeomorphisms
on $\maclH _s(\rr d)$, $\mascS (\rr d)$, $\mascS '(\rr d)$
and on $\maclH _s'(\rr d)$;

\vrum

\item if $\IM (\varrho )>0$ and $s\ge \frac 12$, then
$\mascF _{\! \varrho}$
and $e^{-i\frac \pi 4H_{x,\varrho}}$ are discontinuous on
$\maclH _s(\rr d)$, $\mascS (\rr d)$, $\mascS '(\rr d)$
and on $\maclH _s'(\rr d)$.
\end{enumerate}
The same holds true with $s> \frac 12$, $s\le \frac 12$
and $\maclH _{0,s}$ in place of $s\ge \frac 12$, $s< \frac 12$
and $\maclH _s$ at each occurrence.
\end{prop}

\par

\begin{prop}\label{Prop:IntroFracFourComplOrderRefine}
Let $\varrho \in \mathbf C$. Then the following is true:
\begin{enumerate}
\item if $\IM (\varrho )<0$, then $\mascF _{\! \varrho}$
and $e^{-i\frac \pi 4H_{x,\varrho}}$ are continuous from
$\maclS _{1/2}'(\rr d)$ to $\maclS _{1/2}(\rr d)$, and
$$
\mascF _{\! \varrho} (\maclS _{1/2}'(\rr d)) =
e^{-i\frac \pi 4H_{x,\varrho}} (\maclS _{1/2}'(\rr d))
\subsetneq \maclS _{1/2}(\rr d) \text ;
$$

\item if $\IM (\varrho )>0$, then $\mascF _{\! \varrho}$
and $e^{-i\frac \pi 4H_{x,\varrho}}$ are discontinuous from
$\maclS _{1/2}(\rr d)$ to $\maclS _{1/2}'(\rr d)$, and
$$
\maclS _{1/2}'(\rr d) \subsetneq
\mascF _{\! \varrho} (\maclS _{1/2}(\rr d))
=
e^{-i\frac \pi 4H_{x,\varrho}}(\maclS _{1/2}(\rr d))
\subsetneq
\maclH _{0,1/2}'(\rr d).
$$
%
\end{enumerate}
\end{prop}

\par

By usual inclusion relations for Pilipovi{\'c} spaces, it follows that
Proposition \ref{Prop:IntroFracFourComplOrderRefine}
is a refinement of (2) and (4) in Proposition
\ref{Prop:IntroFracFourComplOrder}. In fact, consider
the inclusions
\begin{align}
\maclS _s(\rr d)
&\subsetneq \Sigma _t(\rr d)\subsetneq \maclS _t(\rr d)
\subsetneq \mascS (\rr d)
\notag
\\[1ex]
&\subsetneq \mascS '(\rr d)
\subsetneq \maclS _t(\rr d)\subsetneq \Sigma _t'(\rr d)
\subsetneq \maclS _s'(\rr d),\quad
\frac 12 \le s<t,
\label{Eq:SpacesInclusions}
\end{align}
between the Schwartz space, its distribution space, 
and all (standard) Fourier invariant Gelfand-Shilov spaces of
functions and ultra-distributions. Then
Proposition 
\ref{Prop:IntroFracFourComplOrderRefine} (1) shows that
if $\IM (\varrho )<0$, then the images of the
spaces in \eqref{Eq:SpacesInclusions}
under $\mascF _{\! \varrho}$
and $e^{-i\frac \pi 4H_{x,\varrho}}$ are strict subspaces
of $\maclS _{1/2}(\rr d)$, the smallest space in
\eqref{Eq:SpacesInclusions}. If instead $\IM (\varrho )>0$,
then Proposition 
\ref{Prop:IntroFracFourComplOrderRefine} (2) shows that
the image of this smallest space
$\maclS _{1/2}(\rr d)$ under $\mascF _{\! \varrho}$
and $e^{-i\frac \pi 4H_{x,\varrho}}$ is a
superspace of $\maclS _{1/2}'(\rr d)$,
the largest space in \eqref{Eq:SpacesInclusions}.

\par

This implies, roughly speaking, that the (standard) spaces
in \eqref{Eq:SpacesInclusions} are disqualified when performing
detailed continuity investigations of the canonical density operator
$e^{-H_{x,\varrho}}$ in statistical physics, and that these spaces
can not be used in continuity investigations of the inverse
$e^{H_{x,\varrho}}$ of that operator. On the other hand, Proposition
\ref{Prop:IntroFracFourComplOrder} (1) shows that Pilipovi{\'c}
spaces and their distribution spaces, which are not
Gelfand-Shilov spaces of functions and distributions, are suitable
for continuity investigations of $e^{-H_{x,\varrho}}$ and
$e^{H_{x,\varrho}}$.

\par

In the most general
case we allow $\varrho \in \cc d$, in which
case $H_{x,\varrho}$ is defined as
$$
H_{x,\varrho}
=
\sum _{j=1}^d\rho _j(x_j^2-\partial _j^2),
$$
and $\mascF _{\! \varrho}$ is a fractional Fourier transform
of the multiple order $\varrho$ (see e.{\,}g. \cite{dGoLue}).
%
%
In Section \ref{sec3} we show
that \eqref{Eq:IntroFracFourHarmOscProp}
and \eqref{Eq:BargApprFracFourTrans}
still hold true for such general $\varrho$.

\par

Finally, in Section \ref{sec4} we apply
our results to deduce certain types of Strichartz estimates. We
recall that Strichartz estimates appears
when finding properties on solutions to Cauchy problems like
the Schr{\"o}dinger equation
\begin{equation}\label{Eq:SchrEqParPot}
\begin{cases}
i\partial _tu-H_xu=F, 
\\[1ex]
u(0,x)=u_0(x),\qquad
(t,x)\in I\times \rr d.
\end{cases}
\end{equation}
Here $I=[0,\infty )$ or $I=[0,T]$ for some $T>0$,
$F$ is a suitable function or (ultra-)distribution on $I\times \rr d$
and $u_0$ is a suitable function or (ultra-)distribution on $\rr d$.
It follows that continuity properties of the propagator
\begin{alignat}{2}
(Ef)(t,x) &\equiv (e^{-itH_x}u_0)(x), &
\quad
(t,x) &\in I\times \rr d,
\label{Eq:DefEROpIntro}
\\[1ex]
(S_{1}F)(t,x)
&= 
\int _0^t (e^{-i(t-s)H_x}F(s,\cdo ))(x)
\, ds, &
\qquad
(t,x) &\in I\times \rr d,
\label{Eq:DefS1ROpIntro}
\intertext{and for}
(S_{2}F)(t,x)
&= 
\int _I (e^{-i(t-s)H_x}F(s,\cdo ))(x)
\, ds, &
\quad
(t,x)&\in I\times \rr d,
\label{Eq:DefS2ROpIntro}
\end{alignat}
are essential when finding estimates for solutions to
\eqref{Eq:SchrEqParPot} (see \cite{CorNic,GinVel,Tag}).
Such estimates are called \emph{Strichartz estimates}
(see also Subsection \ref{subsec1.8} for more details).

\par

In Section \ref{sec4} we deduce Strichartz estimates of
the operators $E$, $S_1$ and $S_2$ when acting
on modulation spaces or Lebesgue spaces with values in
modulation spaces.
For example by straight-forward applications
of Proposition \ref{Prop:ConseqMainThms}, we get
the following result, which is a special case of
Theorem \ref{Thm:StrichEstMod2}
in Section \ref{sec4}.

\par

\begin{thm}\label{Thm:StrichEstMod2Intro}
Let $p,q,r_0\in (0,\infty]$ be such that $q\le p$
and
\begin{equation*}
\frac 1{r_0}>d\left (\frac 1q -\frac 1p\right ).
\end{equation*}
Then $E$ is uniquely extendable to a continuous map
$$
E : M^{p,q}(\rr d)
+
W^{q,p}(\rr d)
\to
L^{r_0}([0,T];M^{q,p}(\rr d))
\bigcap
L^{r_0}([0,T];W^{p,q}(\rr d)).
$$
\end{thm}

\par

Another application of Proposition \ref{Prop:ConseqMainThms}
in combination of the Hardy-Littlewood-Sobolev inequality
leads to the following special case of
Theorem \ref{Thm:StrichEstMod1} in Section \ref{sec4}.

\par

\begin{thm}\label{Thm:StrichEstMod1Intro}
Let $p,p_0,q\in (1,\infty ]$ and
$r_0\in (0,\infty )$ be such that
\begin{equation*}
0\le d\left (\frac 1q-\frac 1p\right )
<1,
\quad
d\left (\frac 1q-\frac 1p\right )
\le 1+\frac 1{r_0}-\frac 1{p_0}.
\end{equation*}
Then $S_1$ and $S_2$ from $C([0,T];M^1(\rr d))$ to
$L^\infty ([0,T];M^1(\rr d))$
are uniquely extendable to
continuous mappings
\begin{alignat*}{2}
S_j &:\, & L^{p_0}([0,T];M^{p,q}(\rr d)+W^{q,p}(\rr d))
&\to
L^{r_0}([0,T];M^{q,p}(\rr d)\bigcap W^{p,q}(\rr d)),
\end{alignat*}
$j=1,2$.
\end{thm}

\par

In Section \ref{sec4}
we also discuss well-posed properties for more
general equations, where $H_x$ in
\eqref{Eq:SchrEqParPot}
is replaced by $\zeta H_{x,\varrho}^r$ for some
$\zeta \in \mathbf C$, $\varrho \in \cc d$
and $r>0$ which satisfy
$\IM (\zeta \varrho _j)>0$ for some $j$.
Here we show that
such equations are ill-posed, not only
in the framework of Schwartz functions
and tempered distributions, but also for
Gelfand-Shilov functions and distributions, 
while the equation is well-posed for
suitable Pilipovi{\'c} spaces and their
distributions.

\par

\section*{Acknowledgement}

\par

The authors are grateful to Marianna Chatzakou,
Maurice de Gosson and Alberto Parmeggiani for valuable
comments.

\par

The first author was supported by Vetenskapsr{\aa}det
(Swedish Science Council), within the project 2019-04890.
The second author is thankful for the research grants
(DST/INSPIRE/04/2016/001507)
and the third author is thankful for the research grants
(DST/INSPIRE/04/2019/001914).

\par

\section{Preliminaries}\label{sec1} 

\par

In this section we recall some basic facts. We start by 
discussing Gelfand-Shilov spaces, then modulation
spaces and thereafter Pilipovi{\'c} spaces and some of
their properties. Then
we recall the Bargmann transform and some of its
properties, and introduce suitable classes
of power series expansions and entire functions on
$\cc d$. Finally, in Subsection \ref{subsec1.8}
we recall some facts on Strichartz estimates.

\par

\subsection{Gelfand-Shilov spaces
and their distribution spaces}\label{subsec1.1}

\par

We start by recalling definitions of Fourier invariant (standard) Gelfand-Shilov
spaces and their distribution spaces (cf. e.{\,}g. \cite{GeSh,Pil2}). 
Let $s\ge 0$ and $h\in \mathbf R_+$ be fixed. Then $\mathcal S_{s,h}(\rr d)$
is the set of all $f\in C^\infty (\rr d)$ such that
\begin{equation*}
\nm f{\mathcal S_{s,h}}\equiv \sup \frac {|x^\beta \partial ^\alpha
f(x)|}{h^{|\alpha | + |\beta |}(\alpha !\, \beta !)^s}
\end{equation*}
is finite. Here the supremum is taken over all $\alpha ,\beta \in
\nn d$ and $x\in \rr d$.

\par

Obviously $\mathcal S_{s,h}\subseteq
\mathscr S$ is a Banach space which increases with $h$ and $s$.

\par

The \emph{Gelfand-Shilov space} $\maclS _s(\rr d)$ ($\Sigma _s(\rr d)$)
of Roumieu type (Beurling type) is the projective limit (projective limit)
of $\mathcal S_{s,h}(\rr d)$ with respect to $h$. This implies that
\begin{equation}\label{GSspacecond1}
\maclS _s(\rr d) = \bigcup _{h>0}\maclS _{s,h}(\rr d)
\quad \text{and}\quad
\Sigma _{s}(\rr d) =\bigcap _{h>0}\maclS _{s,h}(\rr d)
\end{equation}
is a so-called LB-space and Fr{\'e}chet space, respectively, with
semi norms $\nm \cdo{\maclS _{s,h}}$,
$h>0$. 

\medspace

Let $\maclS _{s,h}'(\rr d)$ be the ($L^2$-)dual of
$\maclS _{s,h}(\rr d)$.
If $s\ge \frac 12$ ($s>\frac 12$), then the
\emph{Gelfand-Shilov distribution space}
$\maclS _s'(\rr d)$ ($\Sigma _s'(\rr d)$)
is the projective limit (inductive limit) of
$\mathcal S_{s,h}'(\rr d)$ with respect to $h>0$.
Hence
\begin{equation}\tag*{(\ref{GSspacecond1})$'$}
\maclS _s'(\rr d) = \bigcap _{h>0}\maclS _{s,h}'(\rr d)
\quad \text{and}\quad
\Sigma _{s}'(\rr d) =\bigcup _{h>0}\maclS _{s,h}'(\rr d).
\end{equation}
We remark that \eqref{Eq:SpacesInclusions} is true with dense
embeddings, and that $\maclS _s'(\rr d)$ and
$\Sigma _t'(\rr d)$ are the (strong) duals of
$\maclS _s(\rr d)$ and $\Sigma _t(\rr d)$. On the other hand,
if $s<t\le \frac 12$, then $\maclS _s(\rr d)$ and $\Sigma _t(\rr d)$
are trivially equal to $\{ 0\}$ (cf. \cite{GeSh,Pil1,Pil2}).

\par

From now on we let $\mathscr F$ be the Fourier transform,
given by
$$
(\mathscr Ff)(\xi )= \widehat f(\xi ) \equiv (2\pi )^{-\frac d2}\int _{\rr
{d}} f(x)e^{-i\scal  x\xi }\, dx
$$
when $f\in L^1(\rr d)$. Here $\scal \cdo \cdo$ denotes the
usual scalar product on $\rr d$. The map $\mathscr F$ extends 
uniquely to homeomorphisms on $\mathscr S'(\rr d)$,
$\maclS _s'(\rr d)$ and $\Sigma _s'(\rr d)$, and restricts to
homeomorphisms on
$\mathscr S(\rr d)$, $\maclS _s(\rr d)$ and $\Sigma _s(\rr d)$,
and to a unitary operator on $L^2(\rr d)$.

\medspace

Next we recall some mapping properties of Gelfand-Shilov
spaces under short-time Fourier transforms.
Let $\phi \in \mathscr S(\rr d)$ be fixed. For every $f\in
\mathscr S'(\rr d)$, the \emph{short-time Fourier transform} $V_\phi
f$ is the distribution on $\rr {2d}$ defined by the formula
\begin{equation}\label{defstft}
(V_\phi f)(x,\xi ) =\mathscr F(f\, \overline{\phi (\cdo -x)})(\xi ) =
(f,\phi (\cdo -x)e^{i\scal \cdo \xi}).
\end{equation}
We recall that if $T(f,\phi )\equiv V_\phi f$ when $f,\phi \in \maclS _s(\rr d)$,
then $T$ is uniquely extendable to sequentially continuous mappings
\begin{alignat*}{2}
T\, &:\, & \maclS _s'(\rr d)\times \maclS _s(\rr d) &\to
\maclS _s '(\rr {2d})\bigcap C^\infty (\rr {2d}),
\\[1ex]
T\, &:\, & \maclS _s'(\rr d)\times \maclS _s'(\rr d) &\to
\maclS _s'(\rr {2d}),
\end{alignat*}
and similarly with $\Sigma _s$ in place of $\maclS _s$ at
each occurrence
(cf. \cite{CPRT10,Toft18}).
We also note that $V_\phi f$ takes the form
\begin{equation}\tag*{(\ref{defstft})$'$}
V_\phi f(x,\xi ) =(2\pi )^{-\frac d2}\int _{\rr d}f(y)
\overline {\phi (y-x)}e^{-i\scal y\xi}\, dy
\end{equation}
for admissible $f$. 

\par

There are several characterizations of Gelfand-Shilov
spaces and their distribution spaces, e.{\,}g. 
by suitable estimates of their Fourier
and Short-time Fourier transforms (cf.
\cite{ChuChuKim,GroZim,Toft18}).

\par

\subsection{Spaces of sequences}\label{subsec1.2}

\par

The definitions of Pilipovi{\'c} spaces and spaces of
power series expansions are based on certain spaces
of sequences on $\nn d$, indexed by the extended set
$$
\mathbf R_\flat = \mathbf R_+\bigcup
\sets {\flat _\sigma}{\sigma \in \mathbf R_+},
$$
of $\mathbf R_+$. We extend the ordering relation on
$\mathbf R_+$ to the set $\mathbf R_\flat$, by letting
$$
s_1<\flat _\sigma <s_2
\quad \text{and}\quad \flat _{\sigma _1}<\flat _{\sigma _2}
$$
when $s_1,s_2,\sigma _1,\sigma _2\in \mathbf R_+$,
satisfy
$s_1<\frac 12\le s_2$ and $\sigma _1<\sigma _2$.
(Cf. \cite{Toft18}.)

\par

\begin{defn}\label{Def:SeqSpaces}
Let $s\in \mathbf R_\flat$ and $r,\sigma \in \mathbf R_+$.
\begin{enumerate}
\item The set $\ell _0'(\nn d)$ consists of all formal
sequences $a=\{ a(\alpha )\}_{\alpha \in \nn d}
\subseteq \mathbf C$,
and $\ell _0(\nn d)$ is the set of all $a\in \ell _0'(\nn d)$
such that $a(\alpha )\neq 0$ for at most finite numbers
of $\alpha \in \nn d$;

\vrum

\item The Banach spaces $\ell _{s;r}^\infty (\nn d)$ and
$\ell _{s;r}^{\infty ,*} (\nn d)$
consist of all $a\in \ell _0'(\nn d)$ such that
their corresponding norms
\begin{align*}
\nm a{\ell _{s;r}^\infty}
&=
\begin{cases}
\underset {\alpha \in \nn d}\sup  |a(\alpha )e^{r|\alpha |^{\frac 1{2s}}}|,
& s\in \mathbf R_+,
\\[1ex]
\underset {\alpha \in \nn d} \sup  |a(\alpha )
r^{-|\alpha |} \alpha !^{\frac 1{2\sigma }}|, & s=\flat _\sigma ,
\end{cases}
\intertext{and}
\nm a{\ell _{s;r}^{\infty ,*}}
&=
\begin{cases}
\underset {\alpha \in \nn d}\sup  |a(\alpha )
e^{-r|\alpha |^{\frac 1{2s}}}|, & s\in \mathbf R_+,
\\[1ex]
\underset {\alpha \in \nn d} \sup  |a(\alpha )
r^{-|\alpha |} \alpha !^{-\frac 1{2\sigma }}|, &
s=\flat _\sigma .
\end{cases}
\end{align*}
respectively, are finite;

\vrum

\item The space $\ell _s(\nn d)$ ($\ell _{0,s}(\nn d)$) is the
inductive limit (projective limit) of $\ell _{s;r}^\infty (\nn d)$
with respect to $r>0$, and $\ell _s'(\nn d)$
($\ell _{0,s}'(\nn d)$) is the
projective limit (inductive limit) of
$\ell _{s;r}^{\infty ,*} (\nn d)$
with respect to $r>0$.
\end{enumerate}
\end{defn}

\par

We also let $\nm \cdo {\ell _{0;N}}$ be the semi-norm
\begin{equation}\label{Eq:SeqBasicSemiNorm}
\nm a{\ell _{0;N}} \equiv \sup _{|\alpha |\le N}|a(\alpha )|,
\qquad
a\in \ell _0'(\nn d),
\end{equation}
and $\ell _{0;N}(\nn d)$ be the Banach space with norm
\eqref{Eq:SeqBasicSemiNorm} and
consisting of all $a\in \ell _0'(\nn d)$
such that $a(\alpha )=0$ when $|\alpha |\ge N$. Then
$\ell _0(\nn d)$ is the inductive limit of
$\ell _{0;N}(\nn d)$ with respect
to $N$, and $\ell _0'(\nn d)$ is a Fr{\'e}chet
space under the semi-norms
(Cf. \cite{Toft18}.)

\par

In what follows, $(\cdo ,\cdo )_{\mascH}$ denotes the scalar
product in the Hilbert space $\mascH$.

\par

\begin{rem}\label{Rem:SeqSpaces}
Let $s\in \overline{\mathbf R_\flat}$. Then the duals of
\begin{equation}\label{Eq:DiscPilSpaces}
\ell _{0,s} (\nn d), \quad
\ell _s(\nn d), \quad 
\ell _s'(\nn d)\quad \text{and}\quad
\ell _{0,s}'(\nn d)
\end{equation}
are given by
$$
\ell _{0,s}' (\nn d), \quad
\ell _s'(\nn d), \quad 
\ell _s(\nn d)\quad \text{and}\quad
\ell _{0,s}(\nn d),
$$
respectively, with respect to unique extensions of the form
$(\cdo ,\cdo )_{\ell ^2(\nn d)}$ on $\ell _0(\nn d)\times \ell _0(\nn d)$.
If $s>0$, then $\ell _0(\nn d)$ is dense in $\ell _0'(\nn d)$ and the spaces
in \eqref{Eq:DiscPilSpaces}. (See e.{\,}g. \cite{Toft18})
\end{rem}

\par

\subsection{Pilipovi{\'c} spaces and spaces of power series expansions
on $\cc d$}\label{subsec1.3}

\par

We recall that
the Hermite function of order $\alpha \in \nn d$ is defined by
$$
h_\alpha (x) = \pi ^{-\frac d4}(-1)^{|\alpha |}
(2^{|\alpha |}\alpha !)^{-\frac 12}e^{\frac 12\cdot {|x|^2}}
(\partial ^\alpha e^{-|x|^2}).
$$
It follows that
$$
h_{\alpha}(x)=   ( (2\pi )^{\frac d2} \alpha ! )^{-1}
e^{-\frac 12\cdot {|x|^2}}p_{\alpha}(x),
$$
for some polynomial $p_\alpha$ of order $\alpha$
on $\rr d$,
called the Hermite polynomial of order $\alpha$. 
The Hermite functions are eigenfunctions to the Fourier transform, and
to the Harmonic oscillator
\begin{equation}\label{Eq:HarmOscDef}
H_{x,c}\equiv H_x+c ,\qquad H_x\equiv |x|^2-\Delta _x,
\qquad x\in \rr d,
\end{equation}
which acts on functions and (ultra-)distributions defined
on $\rr d$. Here $c\in \mathbf C$ is fixed. More precisely, we have
\begin{equation}\label{Eq:HarmOscHermFunc}
H_{x,c}h_\alpha = (2|\alpha |+d+c)h_\alpha .
\end{equation}

\par

More generally, for any $c\in \mathbf C$ and
$\varrho =(\varrho _1,\dots ,\varrho _d)\in \cc d$,
we let
\begin{equation}\label{Eq:HarmOscDefExt}
H_{x,\varrho ,c}
\equiv 
\left (
\sum _{j=1}^d \varrho _j(x_j^2-\partial _{x_j}^2)
\right )+c
=
\left (
\sum _{j=1}^d \varrho _jH_{x_j}
\right )+c,
\qquad x\in \rr d.
\end{equation}
Evidently, $H_{x,\varrho ,c}$ is positive definite when
$$
\varrho \in \rr d_+
\quad \text{and}\quad
c>-\sum _{j=1}^d\varrho _j.
$$
For conveniency we put
$H_{x,\varrho ,c}=H_{x,\varrho _0,c}$ when
$\varrho =(\varrho _0,\dots ,\varrho _0)\in \cc d$,
and observe that
$$
H_{x,\varrho ,c} = \varrho H_x +c
\quad \text{when}\quad
\varrho \in \mathbf C.
$$

\par

It is well-known that
the set of Hermite functions is a basis for $\mascS (\rr d)$ 
and an orthonormal basis for $L^2(\rr d)$ (cf. \cite{ReSi}).
In particular, if $f,g\in L^2(\rr d)$, then
$$
\nm f{L^2(\rr d)}^2 = \sum _{\alpha \in \nn d}|c_h(f,\alpha )|^2
\quad \text{and}\quad
(f,g)_{L^2(\rr d)} = \sum _{\alpha \in \nn d}c_h(f,\alpha )
\overline{c_h(g,\alpha )},
$$
where
\begin{align}
f(x) &= \sum _{\alpha \in \nn d}c_h(f,\alpha )h_\alpha (x)
\label{Eq:HermiteExp}
\intertext{is the Hermite seriers expansion of $f$, and}
c_h(f,\alpha ) &= (f,h_\alpha )_{L^2(\rr d)}
\label{Eq:HermiteCoeff}
\end{align}
is the Hermite coefficient of $f$ of order $\alpha \in \rr d$.

\par

We let $\maclH _0'(\rr d)$ be the set of all formal Hermite series
expansions in \eqref{Eq:HermiteExp}, and $\maclA _0'(\cc d)$
be the set of all formal power series expansions
\begin{equation}\label{Eq:PowerSeriesExp}
F(z) = \sum _{\alpha \in \nn d}c(F,\alpha )e_\alpha (z),\qquad
e_\alpha (z)= \frac {z^\alpha}{\sqrt {\alpha !}},\ \alpha \in \nn d,
\end{equation}
on $\cc d$. Then the map
\begin{alignat}{2}
T_{\maclH} \, &: &\, \{ c(\alpha )\} _{\alpha \in \nn d}
&\mapsto \sum _{\alpha \in \nn d} c(\alpha )h_\alpha
\label{Eq:THMap}
\intertext{is bijective from $\ell _0'(\nn d)$ to $\maclH _0'(\rr d)$, and}
T_{\maclA} \, &: &\, \{ c(\alpha )\} _{\alpha \in \nn d}
&\mapsto \sum _{\alpha \in \nn d} c(\alpha )e_\alpha
\label{Eq:TAMap}
\end{alignat}
is bijective from $\ell _0'(\nn d)$ to
$\maclA _0'(\cc d)$. We let the topologies
of $\maclH _0'(\rr d)$ and $\maclA _0'(\cc d)$
be inherited from $\ell _0'(\nn d)$
through the mappings $T_{\maclH}$ and
$T_{\maclA}$, respectively.

\par

\begin{defn}\label{Def:PilPowerSpaces}
Let $s\in \overline{\mathbf R_\flat}$.
\begin{enumerate}
\item The spaces
\begin{equation}\label{Eq:Pilspaces}
\maclH _{0,s}(\rr d), \quad \maclH _s(\rr d), \quad
\maclH _s'(\rr d)
\quad \text{and}\quad
\maclH _{0,s}'(\rr d),
\end{equation}
and their topologies, are the images under
the map $T_{\maclH}$ of the spaces
and their topologies in \eqref{Eq:DiscPilSpaces}, respectively.

\vrum

\item  The spaces
\begin{equation}\label{Eq:AnalPilspaces}
\maclA _{0,s}(\cc d), \quad \maclA _s(\cc d), \quad
\maclA _s'(\cc d)
\quad \text{and}\quad
\maclA _{0,s}'(\cc d),
\end{equation}
and their topologies, are the images under
the map $T_{\maclA}$ of the spaces
and their topologies in \eqref{Eq:DiscPilSpaces}, respectively.
\end{enumerate}
\end{defn}

\par

The spaces $\maclH _s(\rr d)$ and $\maclH _{0,s}(\rr d)$ in Definition
\ref{Def:PilPowerSpaces} are called
\emph{Pilipovi{\'c} spaces} of \emph{Roumieu} respectively
\emph{Beurling types} of order $s$,
and
$\maclH _s'(\rr d)$ and $\maclH _{0,s}'(\rr d)$ are called
\emph{Pilipovi{\'c} distribution spaces} of \emph{Roumieu} respectively
\emph{Beurling types} of order $s$.

\par

There are several characterizations of Pilipovi{\'c} spaces, e.{\,}g.
in terms of estimates of powers of the harmonic oscillator on
the involved functions (see e.{\,}g. 
\cite{AbFeGaToUs,FeGaTo,Toft18}).

\par

\begin{rem}\label{Rem:IdentPilSpaces}
Let $s_1,s_2\in \overline{\mathbf R}_\flat$.
For future references we recall that
\begin{equation}\label{Eq:NonTrivialGS}
\begin{alignedat}{3}
\maclH _{s_1}(\rr d) &= \maclS _{s_1}(\rr d), &
\quad
\maclH _{s_1}'(\rr d) &= \maclS _{s_1}'(\rr d), &
\quad
s_1 &\ge \frac 12,
\\[1ex]
\maclH _{0,s_2}(\rr d) &= \Sigma _{s_2}(\rr d), &
\quad
\maclH _{0,s_2}'(\rr d) &= \Sigma _{s_2}'(\rr d), &
\quad
s_2 &> \frac 12,
\end{alignedat}
\end{equation}
while for the other choices of $s_1,s_2$ we have
\begin{equation}\label{Eq:TrivialGS}
\begin{alignedat}{3}
\maclH _{s_1}(\rr d) &\neq \maclS _{s_1}(\rr d)=\{ 0\} , &
\qquad
s_1 &< \frac 12,
\\[1ex]
\maclH _{0,s_2}(\rr d) &\neq \Sigma _{s_2}(\rr d)=\{ 0\} , &
\qquad
0<s_2 &\le \frac 12,
\end{alignedat}
\end{equation}
and that $\maclH _{s_1}(\rr d)$ and $\maclH _{0,s_2}(\rr d)$
in \eqref{Eq:NonTrivialGS} and \eqref{Eq:TrivialGS}
are dense in $\mascS (\rr d)$.
(See e.{\,}g. \cite{Pil2,Toft18}.)

\par

Hence, any non-trivial Gelfand-Shilov space and
its distribution space, agree with corresponding
Pilipovi{\'c} space and its distribution space. In
particular, Gelfand-Shilov spaces and their
distribution spaces can be characterized in
similar ways as Pilipovi{\'c} spaces and 
their distribution spaces in terms of estimates of
their coefficients in their Hermite function expansions.

\par

In this context we also recall that the Schwartz space
and the set of tempered distributions can be characterized
as
\begin{alignat}{4}
f&\in \mascS (\rr d) &
\quad &\Leftrightarrow & \quad
|c_h(f,\alpha )| &\lesssim \eabs \alpha ^{-N} &
\ &\text{for every}\ N\ge 0
\label{Eq:SchwSpHermChar}
\intertext{and}
f&\in \mascS '(\rr d) &
\quad &\Leftrightarrow & \quad
|c_h(f,\alpha )| &\lesssim \eabs \alpha ^{N} &
\ &\text{for some}\ N\ge 0.
\label{Eq:TempDistHermChar}
\end{alignat}
(See e.{\,}g. \cite{ReSi}). Here and in what follows we
let
$$
\eabs x = (1+|x|^2)^{\frac 12}.
$$
\end{rem}

\par

\subsection{Weight functions}\label{subsec1.4}

\par

Next we recall some facts on weight
functions. A \emph{weight} on $\rr d$ is a positive function $\omega
\in  L^\infty _{loc}(\rr d)$ such that $1/\omega \in  L^\infty _{loc}(\rr d)$.
The set of weights on $\rr d$ is denoted by $\mascP _{\! A}(\rr d)$.
In the sequel we usually assume that $\omega$ is \emph{moderate},
or \emph{$v$-moderate} for some positive function $v \in
 L^\infty _{loc}(\rr d)$. This means that
\begin{equation}\label{moderate}
\omega (x+y) \lesssim \omega (x)v(y),\qquad x,y\in \rr d.
\end{equation}
Here $A\lesssim B$ means that $A\le cB$
for a suitable constant $c>0$, and we
write $A\asymp B$
when $A\lesssim B$ and $B\lesssim A$. 
We note that \eqref{moderate} implies that $\omega$ fulfills
the estimates
\begin{equation}\label{moderateconseq}
v(-x)^{-1}\lesssim \omega (x)\lesssim v(x),\quad x\in \rr d.
\end{equation}
We let $\mascP _E(\rr d)$ be the sets of all moderate
weights on $\rr d$.

\par

In several situations we also deal with
weights which are radial symmetric in each phase space
variable $(x_j,\xi _j)$. The set of such weights is denoted by
$\mascP _{\! A,r}(\rr {2d})$. That is, $\mascP _{\! A,r}(\rr {2d})$
consists of all $\omega \in \mascP _{\! A}(\rr {2d})$
such that $\omega (x,\xi )=\omega _0(\rho )$
for some $\omega _0\in \mascP _{\! A}(\rr d)$,
where $\rho _j=x_j^2+\xi _j^2$.

\par

It can be proved that if $\omega \in \mascP _E(\rr d)$, then
$\omega$ is $v$-moderate for some $v(x) = e^{r|x|}$, provided the
positive constant $r$ is chosen large enough (cf. \cite{Gro3}). In particular,
\eqref{moderateconseq} shows that for any $\omega \in \mascP
_E(\rr d)$, there is a constant $r>0$ such that
\begin{equation}\label{WeightExpEst}
e^{-r|x|}\lesssim \omega (x)\lesssim e^{r|x|},\quad x\in \rr d.
\end{equation}

\par

We also let $\mascP (\rr d)$ be the set of all weights $\omega$
on $\rr d$ such that $\omega$ is moderated by
$v(x,\xi )=(1+|x|+|\xi |)^{r}$, for some $r\ge 0$.
Evidently, $\mascP (\rr d)\subseteq \mascP _E(\rr d)$.

\par

We say that $v$ is
\emph{submultiplicative} if $v$ is \emph{even} and \eqref{moderate}
holds with $\omega =v$. In the sequel, $v$ always stand for a
submultiplicative weight if nothing else is stated.

\par

\subsection{Modulation spaces and Wiener
amalgam spaces}\label{subsec1.5}

\par

Before defining modulation spaces we first address
some notions
on mixed norm spaces of Lebesgue types.
Let $p,q\in (0.\infty ]$ and
$r=\min (p,q)$. For any $f\in L^r_{\loc}(\rr {2d})$, let
\begin{alignat*}{3}
\nm f{L^{p,q}} = \nm f{L^{p,q}(\rr {2d})} &\equiv \nm {g_{1,f,p}}{L^q(\rr d)}, &
\quad &\text{where} & \quad
g_{1,f,p}(\xi ) &\equiv \nm {f(\cdo ,\xi )}{L^p(\rr d)}
\intertext{and}
\nm f{L^{p,q}_*} = \nm f{L^{p,q}_*(\rr {2d})} &\equiv \nm {g_{2,f,q}}{L^p(\rr d)}, &
\quad &\text{where} & \quad
g_{2,f,q}(x) &\equiv \nm {f(x,\cdo )}{L^q(\rr d)}.
\end{alignat*}
We also let $L^{p,q}(\rr {2d})$ and $L^{p,q}_*(\rr {2d})$
be the quasi-Banach spaces which consist of all $f\in L^r_{\loc}(\rr {2d})$
such that $\nm f{L^{p,q}}$ and $\nm f{L^{p,q}_*}$ are finite, respectively.

\par

Let $\phi \in \Sigma _1(\rr d)\setminus 0$, $p,q\in (0.\infty ]$
and $\omega \in \mascP _E(\rr {2d})$. Then the modulation
spaces $M^{p,q}_{(\omega )}(\rr d)$
and $W^{p,q}_{(\omega )}(\rr d)$ consist of
all $f\in \Sigma _1'(\rr d)$ such that $V_\phi f\cdot \omega$
belongs to $L^{p,q}(\rr {2d})$ respectively $L^{p,q}_*(\rr {2d})$.
We equip $M^{p,q}_{(\omega )}(\rr d)$ and $W^{p,q}_{(\omega )}(\rr d)$
with the quasi-norms
\begin{equation}\label{modnorm2}
f\mapsto \nm f{M^{p,q}_{(\omega )}}
\equiv \nm {V_\phi f\cdot \omega}{L^{p,q}}
\quad \text{and}\quad
f\mapsto \nm f{W^{p,q}_{(\omega )}}
\equiv \nm {V_\phi f\cdot \omega}{L^{p,q}_*},
\end{equation}
respectively. For conveniency we also set
$M^p_{(\omega )}=M^{p,p}_{(\omega )}$, and
remark that $M^{p,q}_{(\omega )}(\rr d)$ 
is one of the most common types of modulation spaces. We also set
$M^{p,q}=M^{p,q}_{(\omega)},$ and $M^{p}=M^{p}_{(\omega)}$ when
$\omega =1$, and similarly for $W^{p,q}_{(\omega)}$ spaces.

\par

Modulation spaces with $\omega \in \mascP (\rr {2d})$
were introduced by Feichtinger in \cite{Fei1}. The theory was
thereafter extended and generalized in several ways (see e.{\,}g.
\cite{Fei6,FeiGro1,FeiGro2,GaSa}).
%
%
%
In the following proposition we list some basic properties for modulation spaces.
We refer to
\cite{Fei1,FeiGro1,GaSa, Gc1,Toft2} for the proof.

\par

\begin{prop}\label{p1.4B}
Let $r\in (0,1]$, $p,p_j,q_j\in (0,\infty ]$ and $\omega ,\omega _j,v\in
\mascP  _{E}(\rr {2d})$, $j=1,2$, be such that $r\le p,q$,
$p _1\le p _2$, $q_1\le q_2$,  $\omega _2\lesssim \omega _1$, and
let $\omega$ be $v$-moderate. Then the following is true:
\begin{enumerate}
%
%
\item[{\rm{(1)}}] $\Sigma _1(\rr d)\subseteq
M^{p,q}_{(\omega )}(\rr d), W^{p,q}_{(\omega )}(\rr d)
\subseteq \Sigma _1'(\rr d)$ with
continuous inclusions. If in addition $p,q<\infty$, then
$\Sigma _1(\rr d)$ is dense in $M^{p,q}_{(\omega )}(\rr d)$
and $W^{p,q}_{(\omega )}(\rr d)$.
If, more restricted, $\omega \in \mascP (\rr {2d})$,
then similar facts hold true with $\mascS$ in place of $\Sigma _1$
at each occurrence;

\vrum

\item[{\rm{(2)}}] if $\phi \in M^r_{(v)}(\rr d)\setminus 0$, then
$f\in M^{p,q}_{(\omega )}(\rr d)$, if and only if
$\nm {V_\phi f\cdot \omega}{L^{p,q}}$ is finite, and
$f\in W^{p,q}_{(\omega )}(\rr d)$, if and only if
$\nm {V_\phi f\cdot \omega}{L^{p,q}_*}$ is finite.
In particular, $M^{p,q}_{(\omega )}(\rr d)$ and $W^{p,q}_{(\omega )}(\rr d)$
are independent
of the choice of $\phi \in M^r_{(v)}(\rr d)\setminus 0$.
Moreover, different choices of $\phi$ in \eqref{modnorm2}
give rise to equivalent quasi-norms;

\vrum

\item[{\rm{(3)}}] $M^{p _1,q_1}_{(\omega _1)}(\rr d)\subseteq
M^{p _2,q_2}_{(\omega _2)}(\rr d)$
and $W^{p _1,q_1}_{(\omega _1)}(\rr d)\subseteq
W^{p _2,q_2}_{(\omega _2)}(\rr d)$;

\vrum

\item[{\rm{(4)}}] if $\omega _0(\xi ,x) = \omega (-x,\xi )$, then
$\mascF$ is a homeomorphism from $M^{p,q}_{(\omega )}(\rr d)$
to $W^{q,p}_{(\omega _0)}(\rr d)$.
\end{enumerate}
\end{prop}

\par

\begin{rem}\label{Rem:GeneralModSpaces}
In the framework of Pilipovi{\'c} distribution spaces, the definition of
modulation spaces are extended in \cite{Toft18}
to include more general weights (which are not necessary moderate).
In these approaches the window function $\phi$ is fixed and equal to
the Gaussian $\phi (x)=\pi ^{-\frac d4}e^{-\frac 12|x|^2}$. For any
$\omega \in \mascP _{\! A}(\rr {2d})$ and $p,q\in (0,\infty ]$,
the modulation spaces $M^{p,q}_{(\omega )}(\rr d)$ and
$W^{p,q}_{(\omega )}(\rr d)$ consist of all $f\in \maclH _{\flat _1}'(\rr d)$
such that corresponding quasi-norms in \eqref{modnorm2}
are finite. It is proved in \cite{Toft18} that $M^{p,q}_{(\omega )}(\rr d)$ and
$W^{p,q}_{(\omega )}(\rr d)$ are \emph{quasi-Banach spaces}.
If in addition $p,q\ge 1$, then these spaces are \emph{Banach spaces}.
\end{rem}

\par

\subsection{Spaces of entire functions and
the Bargmann transform}\label{subsec1.6}

\par

Let $\Omega \subseteq \cc d$ be open.
Then
$A(\Omega )$ denotes the set of all
analytic functions in $\Omega$.

\par

Next we recall some properties of the Bargmann
transform (cf. \cite{Ba1,Ba2}).
We set
\begin{gather*}
\scal zw = \sum _{j=1}^dz_jw_j
\quad \text{and}\quad (z,w) = \scal z{\overline w}
,\quad \text{when}
\\[1ex]
z=(z_1,\dots ,z_d) \in \cc d
\quad \text{and} \quad
w=(w_1,\dots ,w_d)\in \cc d,
\end{gather*}
and otherwise $\scal \cdo \cdo $ denotes the
duality between test
function spaces and their corresponding duals.
The Bargmann transform $\mathfrak V_df$ of $f\in L^2(\rr d)$
is defined by the formula
\begin{equation}\label{Eq:BargmTransf}
(\mathfrak V_df)(z) =\pi ^{-\frac d4}\int _{\rr d}\exp \Big ( -\frac 12(\scal
z z+|y|^2)+2^{\frac 12}\scal zy \Big )f(y)\, dy
\end{equation}
(cf. \cite{Ba1}). We note that if $f\in
L^2(\rr d)$, then the Bargmann transform
$\mathfrak V_df$ of $f$ is the entire function on
$\cc d$, given by
$$
(\mathfrak V_df)(z) =\int _{\rr d}\mathfrak A_d(z,y)f(y)\, dy,
$$
or
\begin{equation}\label{bargdistrform}
(\mathfrak V_df)(z) =\scal f{\mathfrak A_d(z,\cdo )},
\end{equation}
where the Bargmann kernel $\mathfrak A_d$ is given by
$$
\mathfrak A_d(z,y)=\pi ^{-\frac d4} \exp \Big ( -\frac 12(\scal
zz+|y|^2)+2^{\frac 12}\scal zy\Big ).
$$
Evidently, the right-hand side in \eqref{bargdistrform}
makes sense when $f\in \maclS _{1/2}'(\rr d)$ and
defines an element in $A(\cc d)$, since
$y\mapsto \mathfrak A_d(z,y)$ can be interpreted
as an element in $\maclS _{1/2} (\rr d)$ with values
in $A(\cc d)$.

\par

It was proved in \cite{Ba1} that $f\mapsto
\mathfrak V_df$ is a bijective
and isometric map  from $L^2(\rr d)$ to the Hilbert
space $A^2(\cc d) \equiv B^2(\cc d)\cap A(\cc d)$,
where $B^2(\cc d)$ consists of all
measurable functions $F$ on $\cc  d$ such that
\begin{equation}\label{A2norm}
\nm F{B^2}
\equiv
\Big ( \int _{\cc d}|F(z)|^2d\mu (z)
\Big )^{\frac 12}<\infty .
\end{equation}
Here $d\mu (z)=\pi ^{-d} e^{-|z|^2}\, d\lambda (z)$, where 
$d\lambda (z)$ is the
Lebesgue measure on $\cc d$. We recall
that $A^2(\cc d)$ and $B^2(\cc d)$
are Hilbert spaces, where the scalar product are given by
\begin{equation}\label{A2scalar}
(F,G)_{B^2}
\equiv
\int _{\cc d} F(z)\overline {G(z)}\, d\mu (z),
\quad F,G\in B^2(\cc d).
\end{equation}
If $F,G\in A^2(\cc d)$, then we set
$\nm F{A^2}=\nm F{B^2}$
and $(F,G)_{A^2}=(F,G)_{B^2}$.

\par

In \cite{Ba1} it is also proved that
\begin{equation}\label{BargmannHermite}
\mathfrak V_dh_\alpha  = e_\alpha ,\quad \text{where}\quad
e_\alpha (z)\equiv \frac {z^\alpha}{\sqrt {\alpha !}},\quad z\in \cc d .
\end{equation}
In particular, the Bargmann transform maps the orthonormal basis
$\{ h_\alpha \}_{\alpha \in \nn d}$ in $L^2(\rr d)$ bijectively into the
orthonormal basis $\{ e_\alpha \}_{\alpha \in \nn d}$ of monomials
in $A^2(\cc d)$.

\par

For general $f\in \maclH _0'(\rr d)$ we now set
\begin{equation}\label{Eq:GeneralBargmannDef}
\mathfrak V_df \equiv (T_{\maclA}\circ T_{\maclH}^{-1})f,
\qquad f\in \maclH _0'(\rr d),
\end{equation}
where $T_{\maclH}$ and $T_{\maclA}$ are given by 
\eqref{Eq:THMap}
and \eqref{Eq:TAMap}. It follows from
\eqref{BargmannHermite} that $\mathfrak V_df$ in
\eqref{Eq:GeneralBargmannDef} agrees with
$\mathfrak V_df$ in \eqref{Eq:BargmTransf} when
$f\in L^2(\rr d)$, and that this is the only way to extend
the Bargmann transform continuously to a
continuous map from $\maclH _0'(\rr d)$ to
$\maclA _0'(\cc d)$. From these observations
and definitions, we get the following. The details
are left for the reader.

\par

\begin{prop}
Let $s\in \overline {\mathbf R_\flat}$. Then $\mathfrak V_d$
is a homeomorphism from $\maclH _0'(\rr d)$
to $\maclA _0'(\cc d)$, and restricts to homeomorphisms
from the spaces in \eqref{Eq:Pilspaces} to the spaces
in \eqref{Eq:AnalPilspaces}, respectively.
\end{prop}

\par

It follows that if $f,g\in L^2(\rr d)$ and
$F,G\in A^2(\cc d)$, then
\begin{equation}\label{Scalarproducts}
\begin{aligned}
(f,g)_{L^2(\rr d)} &= \sum _{\alpha \in \nn d}c_h(f,\alpha ) \overline {c_h(g,\alpha )},
\\[1ex]
(F,G)_{A^2(\cc d)} &= \sum _{\alpha \in \nn d}c(F,\alpha ) \overline {c(G,\alpha )}.
\end{aligned}
\end{equation}

\par

By the definitions we get the following proposition
on duality for Pilipovi{\'c}
spaces and their Bargmann images. The details are left
for the reader.

\par

\begin{prop}
Let $s_1\in \mathbf R_\flat$ and $s_2\in \overline{\mathbf R}_\flat$. Then 
the form $(\cdo ,\cdo )_{L^2(\rr d)}$ on
$\maclH _0(\rr d)\times \maclH _0(\rr d)$ is uniquely extendable to
sesqui-linear forms on
\begin{alignat*}{2}
&\maclH _{s_2}'(\rr d)\times \maclH _{s_2}(\rr d),
&\quad
&\maclH _{s_2}(\rr d)\times \maclH _{s_2}'(\rr d),
\\[1ex]
&\maclH _{0,s_1}'(\rr d)\times \maclH _{0,s_1}(\rr d)
&\quad \text{and on}\quad
&\maclH _{0,s_1}(\rr d)\times \maclH _{0,s_1}'(\rr d).
\end{alignat*}
The (strong) duals of $\maclH _{s_2}(\rr d)$ and $\maclH _{0,s_1}(\rr d)$
are equal to $\maclH _{s_2}'(\rr d)$ and $\maclH _{0,s_1}'(\rr d)$,
respectively, through the form $(\cdo ,\cdo )_{L^2(\rr d)}$.

\par

The same holds true if the spaces in \eqref{Eq:Pilspaces}
and the form $(\cdo ,\cdo )_{L^2(\rr d)}$ are replaced by
corresponding spaces in \eqref{Eq:AnalPilspaces} and the form
$(\cdo ,\cdo )_{A^2(\cc d)}$, at each occurrence.
\end{prop}

\par

If $s\in \overline {\mathbf R_\flat}$, $f\in \maclH _s(\rr d)$, $g\in \maclH _s'(\rr d)$,
$F \in \maclA _s(\cc d)$ and $G \in \maclA _s'(\cc d)$, then $(f,g)_{L^2(\rr d)}$
and $(F,G)_{A^2(\cc d)}$ are defined by the formula \eqref{Scalarproducts}.
It follows that
\begin{equation}\label{ScalarproductsRel}
c_h(f,\alpha ) = c(F,\alpha )
\quad \text{when}\quad
F=\mathfrak V_df ,\ G=\mathfrak V_dg .
\end{equation}
holds for such choices of $f$ and $g$.

\par

\begin{rem}\label{Rem:AnalSpacesIdent}
In \cite{FeGaTo,Toft18}, the spaces in \eqref{Eq:AnalPilspaces}, contained in
$\maclA _{0,\flat _1}'(\cc d)=A(\cc d)$ are identified as canonical
spaces of analytic functions. For example it is here shown that
if $\sigma _1>0$ and $\sigma _2>1$, then
\begin{align*}
\maclA _{\flat _{\sigma _1}}(\cc d) &=
\sets {F\in A(\cc d)}{|F(z)|\lesssim e^{r|z|^{\frac {2\sigma _1}{\sigma _1 +1}} }
\ \text{for some $r>0$}}
\intertext{and}
\maclA _{\flat _{\sigma _2}}'(\cc d) &=
\sets {F\in A(\cc d)}{|F(z)|\lesssim e^{r|z|^{\frac {2\sigma_2}{\sigma _2 -1}} }
\ \text{for every $r>0$}}.
\end{align*}
\end{rem}

\par

\subsection{Fractional Fourier transforms}
\label{subsec1.7}

\par

We recall that 
%
(multiple ordered)
fractional Fourier transform $\mascF _{\! \varrho}$ with respect to
$\varrho =(\varrho _1,\dots ,\varrho _d)\in \rr d$ is the operator
with kernel given by
$K_{d,\varrho}$ in \eqref{Eq:FracFourKernel} and
\eqref{Eq:FracFourKernelDim1} (see e.{\,}g. \cite{dGoLue}).
Evidently,
\begin{equation}
\label{Eq:FractionalFourierMultipleOrder}
\mascF _{\! \varrho} = \mascF _{\! \varrho _1}
\otimes \cdots \otimes
\mascF _{\! \varrho _d},
\qquad
\varrho =(\varrho _1,\dots ,\varrho _d)\in \rr d,
\end{equation}
and it follows that $\mascF _{\! \varrho}$ makes sense
as homeomorphisms on
$$
\maclH _{0,s}(\rr d),\quad
\maclH _s(\rr d),\quad
\mascS (\rr d),\quad
\mascS '(\rr d),\quad
\maclH _s'(\rr d)
\quad \text{and}\quad
\maclH _{0,s}'(\rr d),
$$
and to a unitary operators on $L^2(\rr d)$
(see e.{\,}g. \cite{Ba1,Toft18}). For conveniency we
put $\mascF _{\! \varrho _0}=\mascF _{\! \varrho}$ when
$\varrho =(\varrho _0,\dots ,\varrho _0)\in \rr d$.

\par

\begin{rem}\label{Rem:FracFour}
Apart from the cases when $\varrho$ in  $\mascF _{\! \varrho}$
are (multiple) integers, the formula for the fractional Fourier transform
might not look like a visual eye candy, because the kernel
$K_\varrho (\xi ,x)$ is rather complex. However
the formula appear naturally by using suitable changes of
symplectic coordinates in quantum mechanics.
In fact, the symplectic map which rotates $(x,\xi )$ in
the phase space $\rr {2d}$ with angle $-\frac \pi 2$
into $(\xi ,-x)$, induces that the observables in quantum mechanics
(which are operators) should
be conjugated by the Fourier transform $\mascF =\mascF _1$ (see
e.{\,}g. \cite{CorGroNicRod,CorNic,dGo1,dGo2,dGoLue}).
It might then be natural to define the fractional Fourier transform
$\mascF _{\! \varrho}$ to be the operator which should conjugate the
quantum observables, when the phase space is rotated with the
angle $-\varrho \frac \pi 2$. That is,
\begin{center}
\begin{tabular}{cccc}
\emph{Symplectic map} & & & \emph{Conjugation of quantum observables}
\\[0.5ex]
\hline
&&&
\\[-1ex]
Rotation with angle \hspace{0.07cm} $-\frac \pi 2$
& & &
$\mascF _1$
\\[1ex]
Rotation with angle $-\varrho \frac \pi 2$
& & &
$\mascF _{\! \varrho}$
\end{tabular}
\end{center}

\par

This gives a unique definition of $\mascF _{\! \varrho}$, and after some
computations, it follows that the kernel of $\mascF _{\! \varrho}$ is given
by $K_{d,\varrho}$ in \eqref{Eq:FracFourKernel} and
\eqref{Eq:FracFourKernelDim1}.
An equivalent approach and which leads to the same
formulae, consist of using the metaplectic
representation of symplectic group and considering metaplectic
operators. (See e.{\,}g. \cite{CorGroNicRod,CorNic,dGo1,dGo2,dGoLue} for
more facts on metaplectic representations and corresponding operators.)

\par

A slightly equivalent way to reach the fractional Fourier transform
consists of investigating mapping properties of the Bargmann
transform, $\mathfrak V_d$. It is proved already in \cite{Ba1}
that the Bargmann image of $\widehat f$ is given by
$$
(\mathfrak V_d(\mascF f))(z) = (\mathfrak V_d\widehat f)(z)
=
(\mathfrak V_d\widehat f)(-iz)
=
(\mathfrak V_d\widehat f)(e^{-i\frac \pi 2}z).
$$
That is, the Bargmann image of $\widehat f$ is obtained by
retrieving corresponding image of $f$ and then (again) rotating
the argument with the angle $-\frac \pi 2$. In particular, the Fourier
transform can be evaluated as
$$
\mascF _1 = \mathfrak V_d^{-1}\circ U_1 \circ \mathfrak V_d,
\qquad
(U_1 F)(z)= F(e^{-i\frac \pi 2}z).
$$
It might then be natural to define the fractional Fourier transform as
$$
\mascF _{\! \varrho} = \mathfrak V_d^{-1}\circ U_\varrho \circ \mathfrak V_d,
\qquad
(U_\varrho F)(z)= F(e^{-i\varrho \frac \pi 2}z),
$$
and a straight-forward computations show that we attain
the same formula of the kernel $K_{d,\varrho}$ of
$\mascF _{\! \varrho}$ as before.
%
%
%
%
\end{rem}

\par

Due to the previous remark, the Bargmann
image of $\mascF _{\! \varrho}$ in \eqref{Eq:FracFourKernel},
\eqref{Eq:FracFourKernelDim1} and
\eqref{Eq:FractionalFourierMultipleOrder}
takes the form
\begin{equation}
\label{Eq:FracFourOnBargm}
(\mathfrak V_d(\mascF _{\! \varrho} f))(z)
=
(\mathfrak V_df)
(e^{-\frac i2{\pi }\varrho _1}z_1,\dots
,e^{-\frac i2{\pi }\varrho _d}z_d),\qquad
f \in \maclH _0'(\rr d).
\end{equation}

\par

Let $\phi (x)= \pi ^{-\frac d4}e^{-\frac 12 |x|^2}$.
Then we recall that the
Bargmann transform and the short-time Fourier transform
can be linked as
\begin{equation}\label{Eq:BargmannSTFT}
\begin{aligned}
V_\phi f(x,\xi ) &= (2\pi )^{-\frac d2}e^{-\frac 14|z|^2}
e^{-\frac i2\scal x\xi}
(\mathfrak V_df)(2^{-\frac 12}\overline z),
\\[1ex]
z &= x+i\xi ,\ x,\xi \in \rr d
\end{aligned}
\end{equation}
(see (1.28) in \cite{Toft10}). A combination
of \eqref{Eq:FracFourOnBargm} and
\eqref{Eq:BargmannSTFT} gives that if
$\varrho \in \rr d$,
\begin{equation}\label{Eq:CoordRotations}
\begin{alignedat}{2}
R_{1,\varrho _j}(x_j,\xi _j)
&=
(\cos \theta _j)x_j + (\sin \theta _j)\xi _j , &&
\\[1ex]
R_{2,\varrho _j}(x_j,\xi _j)
&=
-(\sin \theta _j)x_j +(\cos \theta _j)\xi _j, &
\qquad \theta _j
&=
\textstyle{\frac {\pi \varrho _j} 2}
\\[1ex]
R_{k,\varrho}(x,\xi )
&=
\big ( R_{k,\varrho _1}(x_1,\xi _1),\dots
,R_{k,\varrho_d}(x_d,\xi _d)
\big ), &
\qquad
k &= 1,2,
\\[1ex]
A_{d,\varrho }(x,\xi )
&=
(R_{1,\varrho }(x,\xi ),R_{2,\varrho }(x,\xi )), &&
\\[1ex]
U_{d,\varrho }(x,\xi )
&=
R_{1,\varrho }(x,\xi )+iR_{2,\varrho }(x,\xi ), &&
\end{alignedat}
\end{equation}
then
\begin{equation}\label{Eq:STFTFracFour}
\begin{aligned}
(V_\phi (\mascF _{\! \varrho} f))(x,\xi )
&= (2\pi )^{-\frac d2}e^{-\frac 14|z|^2}e^{-\frac i2\scal x\xi}
(\mathfrak V_df)(2^{-\frac 12}\overline {U_{d,\varrho} (z)})
\\[1ex]
&=
e^{i\frac 14\Phi _\varrho (x,\xi )}
V_\phi f(A_{d,\varrho} (x,\xi )),
\quad
z=x+i\xi \in \cc d.
\end{aligned}
\end{equation}

\par

For conveniency we set
$$
A_{d,\varrho _0}(x,\xi )=A_{d,\varrho}(x,\xi )
\quad \text{and}\quad
U_{d,\varrho _0}(x,\xi )=U_{d,\varrho}(x,\xi )
$$
when
$$
\varrho =(\varrho _0,\dots ,\varrho _0)\in \rr d
\quad \text{and}\quad
x,\xi \in \rr d.
$$
It it then clear that \eqref{Eq:STFTFracFour} still hold
true when $\mascF _{\! \varrho}$ is the fractional
Fourier transform on $\maclH _0'(\rr d)$ of order
$\varrho \in \mathbf R$.


\par

As in \cite{Toft18} we observe that if,
more generally, $\varrho \in \cc d$, then the map
$$
f\mapsto
\big (
z\mapsto (\mathfrak V_df
(e^{-\frac i2\pi \varrho _1}z_1,\dots ,e^{-\frac i2\pi \varrho _d}z_d)
\big )
$$
makes sense as a homeomorphism from $\maclH _0'(\rr d)$ into
$\maclA _0'(\cc d)$. In similar ways as in \cite{Toft18},
we define the fractional Fourier transform
$\mascF _{\! \varrho}$ by \eqref{Eq:FracFourOnBargm} when
$\varrho \in \cc d$.
Then it follows that still $\mascF _{\! \varrho}$ is continuous on
$\maclH _0'(\rr d)$.

\par

\subsection{Strichartz estimates}\label{subsec1.8}

\par

We recall that for a linear operator $R$ acting on
suitable functions
or (ultra-)distributions on $\rr d$, Strichartz estimates appears
when finding properties on solutions to Cauchy problems like
the generalized inhomogeneous Schr{\"o}dinger equation
\begin{equation}\label{Eq:GenSchrEqParPot}
\begin{cases}
i\partial _tu-Ru=F, 
\\[1ex]
u(0,x)=u_0(x),\qquad
(t,x)\in I\times \rr d.
\end{cases}
\end{equation}
(See e.{\,}g. \cite{Tag} and the references therein.)
Here $I=[0,\infty )$ or $I=[0,T]$ for some $T>0$,
$F$ is a suitable function or (ultra-)distribution on $I\times \rr d$,
$R$ is a linear operator acting on functions or distributions on $\rr d$,
and $u_0$ is a suitable function or (ultra-)distribution on $\rr d$.
The solution of \eqref{Eq:GenSchrEqParPot}
is formally given by
\begin{equation}\label{Eq:FormalSol}
u(t,x) = (e^{-itR}u_0)(x)-i\int _0^t (e^{i(t-s)R}F(t,\cdo ))(x)\, ds.
\end{equation}
In particular it follows that continuity properties of the
propagator
\begin{alignat}{2}
(E_Rf)(t,x) &\equiv (e^{-itR}u_0)(x), &
\quad
(t,x) &\in I\times \rr d.
\label{Eq:DefEROp}
\intertext{as well as for the operator}
(S_{1,R}F)(t,x)
&= 
\int _0^t (e^{-i(t-s)R}F(s,\cdo ))(x)
\, ds, &
\qquad
(t,x) &\in I\times \rr d,
\label{Eq:DefS1ROp}
\intertext{are essential for finding estimates for solutions to
\eqref{Eq:GenSchrEqParPot}. We observe that
the $L^2(I\times \rr d)$ adjoint $E_R^*$ of $E_R$
is given by}
(E_R^*F)(x) &= \int _I (e^{isR}F(s,\cdo ))(x)\, ds, &
\quad
x &\in \rr d,
\label{Eq:DefERAdjOp}
\intertext{and that the composition $E_R\circ E_R^*$
of $E_R$ and $E_R^*$ is the operator $S_{2,R}$,
similar to $S_{1,R}$, and given by}
(S_{2,R}F)(t,x)
&= 
\int _I (e^{-i(t-s)R}F(s,\cdo ))(x)
\, ds, &
\quad
(t,x)&\in I\times \rr d.
\label{Eq:DefS2ROp}
\end{alignat}
We recall that continuity properties of $E_R$ (or $E_R^*$)
are strongly linked to continuity properties for $S_{2,R}$
(see \cite{GinVel}). Estimates on the operator $E_R$ in
\eqref{Eq:DefEROp} is called \emph{homogeneous Strichartz
estimates}, while estimates for $S_{1,R}$ in \eqref{Eq:DefS1ROp},
or even for $S_{2,R}$ in \eqref{Eq:DefS2ROp}, are called
\emph{inhomogeneous Strichartz estimates}.

\par

In our situation, the operator $R$ is given by the
operator
$H_{x,\varrho ,c}$ for some $c\in \mathbf C$
and $\varrho \in \rr d$ or $\varrho \in \cc d$,
and for such choice of $R$, we put
$E=E_R$ and $S_j=S_{j,R}$, $j=1,2$. Hence
\begin{alignat}{2}
(Ef)(t,x) &\equiv (e^{-itH_{x,\varrho ,c}}u_0)(x), &
\quad
(t,x) &\in I\times \rr d,
\tag*{(\ref{Eq:DefEROp})$'$}
\\[1ex]
(E^*F)(x) &= \int _I (e^{isH_{x,\varrho ,c}}F(s,\cdo ))(x)\, ds, &
\quad
x &\in \rr d.
\tag*{(\ref{Eq:DefERAdjOp})$'$}
\\[1ex]
(S_{1}F)(t,x)
&= 
\int _0^t (e^{-i(t-s)H_{x,\varrho ,c}}F(s,\cdo ))(x)
\, ds, &
\quad
(t,x) &\in I\times \rr d,
\tag*{(\ref{Eq:DefS1ROp})$'$}
\intertext{and}
(S_{2}F)(t,x)
&= 
\int _I (e^{-i(t-s)H_{x,\varrho ,c}}F(s,\cdo ))(x)
\, ds, &
\quad
(t,x)&\in I\times \rr d.
\tag*{(\ref{Eq:DefS2ROp})$'$}
\end{alignat}

\par

In Section \ref{sec4} we deduce continuity properties for the
operators $E$ when acting on modulation spaces, and for
$S_1$ and $S_2$ when acting on Lebesgue spaces with
values in modulation spaces.

%
%

\par

\section{Powers of generalized
harmonic oscillator propagators
on Pilipovi{\'c} spaces}\label{sec2}

\par

In this section we show that powers of harmonic
oscillators, $H_{x,c}$, or more generally
$H_{x,\varrho ,c}$, are continuous on Pilipovi{\'c}
spaces. If in addition $H_{x,\varrho ,c}$ is
injective, then we show that powers of
$H_{x,\varrho ,c}$ are in fact homeomorphisms
on Pilipovi{\'c} spaces and their distribution
spaces. We also consider harmonic oscillator propagators and
deduce homeomorphism properties of such operators on
Pilipovi{\'c} spaces.

\par

In the last part we show that
powers of harmonic oscillators
are continuous on Hilbert modulation spaces of the form
$M^{2,2}_{(\vartheta _r)}(\rr d)$, where $\vartheta _r(x,\xi )=\eabs {(x,\xi )}^r$.

\par

\subsection{Continuity of powers of $H_{x,c}$ and
their propagators}

\par

For any
$c\in \mathbf C$, it follows that
$H_{x,c}$ is continuous on
$\maclH _0(\rr d)$, and that
\begin{equation}\label{Eq:HermExpPowHarmOsc}
H_{x,c}f(x)
=
\sum _{\alpha \in \nn d}
\left (
2|\alpha |+d+c
\right )
c_h(f,\alpha )h_\alpha (x),
\end{equation}
when $f\in \maclH _0(\rr d)$
is given by \eqref{Eq:HermiteExp}.
By duality it follows that $H_{x,c}$
on $\maclH _0(\rr d)$ is uniquely
extendable to a continuous map on
$\maclH _0'(\rr d)$, and that
\eqref{Eq:HermExpPowHarmOsc} still holds true
when $f\in \maclH _0'(\rr d)$ is given by 
\eqref{Eq:HermiteExp}.

\par

In the same way it follows that if $r\ge 0$
is real, then
\begin{alignat}{4}
H_{x,c} ^r \, &: & \ & \sum _{\alpha \in \nn 
d}c_h(f,\alpha )h_\alpha &
&\mapsto & &
\sum _{\alpha \in \nn d}(2|\alpha |+d+c)^r
c_h(f,\alpha )h_\alpha
\label{Eq:HarmonicPowers}
\intertext{and}
e^{\zeta H_{x,c} ^r} \, &: & \ &
\sum _{\alpha \in \nn 
d}c_h(f,\alpha )h_\alpha &
&\mapsto & & \sum _{\alpha \in \nn d}e^{
(2|\alpha |+d+c)^r}
c_h(f,\alpha )h_\alpha
\label{Eq:HarmonicPowersPropagator}
\end{alignat}
are continuous on $\maclH _0(\rr d)$, and uniquely 
extendable to continuous mappings on
$\maclH _0'(\rr d)$. More generally, if
$r\in \mathbf R$, and
\begin{alignat}{3}
c&\in \mathbf C\setminus
\sets{-2n-d}
{n\in \mathbf N} &
\quad &\text{and} &\quad r&\in \mathbf R
\label{Eq:Condcr1}
\intertext{or}
c&\in \mathbf C &
\quad &\text{and} &\quad r&\in
\overline {\mathbf R}_+ ,
\label{Eq:Condcr2}
\end{alignat}
then the operators \eqref{Eq:HarmonicPowers}
and \eqref{Eq:HarmonicPowersPropagator}
are continuous on $\maclH _0(\rr d)$ and on
$\maclH _0'(\rr d)$. The following result extends
these continuity properties to other
Pilipovi{\'c} spaces.



\par

\begin{prop}\label{Prop:BargPilSpaces}
Let $\zeta \in \mathbf C$, $r\in \mathbf R$ and
$c\in \mathbf C$ be as in
\eqref{Eq:Condcr1} or as in \eqref{Eq:Condcr2},
and let $s,s_1,s_2\in \overline {\mathbf R_{\flat}}$
be such that $0<s_1\le \frac 1{2r}$ and
$s_2< \frac 1{2r}$. Then the following is true:
\begin{enumerate}
\item the map \eqref{Eq:HarmonicPowers}
on $\maclH _0'(\rr d)$ restricts
to continuous mappings on $\mascS (\rr d)$,
$\mascS '(\rr d)$
and on the spaces in \eqref{Eq:Pilspaces}.
If \eqref{Eq:Condcr1} holds true,
then these mappings are homeomorphisms;

\vrum

\item the map \eqref{Eq:HarmonicPowersPropagator} on
$\maclH _0'(\rr d)$ restricts
to homeomorphisms on
\begin{equation}\label{Eq:Pilspaces2}
\maclH _{0,s_1}(\rr d),\quad \maclH _{s_2}(\rr d),
\quad \maclH _{s_2}'(\rr d)
\quad \text{and on}\quad
\maclH _{0,s_1}'(\rr d).
\end{equation}
If in addition $\zeta \in i\mathbf R$, then the map
\eqref{Eq:HarmonicPowersPropagator} is homeomorphic
on $\mascS (\rr d)$, $\mascS '(\rr d)$ and the spaces
in \eqref{Eq:Pilspaces}
\end{enumerate}
\end{prop}

\par

\begin{proof}
The assertion (1) follows from \eqref{Eq:HarmonicPowers},
the definitions of the spaces in
\eqref{Eq:Pilspaces} and the fact that
\begin{equation}\label{Eq:HarmonicPowersSeq}
\{ c(\alpha ) \}_{\nn d}
\mapsto
\{ (2|\alpha |+d+c)^rc(\alpha ) \}_{\nn d}
\end{equation}
is continuous on the sequence spaces in Definition \ref{Def:SeqSpaces} (3).
The homeomorphism property in the case $c\neq -d$ follows from the fact that
the map \eqref{Eq:HarmonicPowersSeq} is then a continuous bijection.

\par

In the same way, (2) follows from \eqref{Eq:HarmonicPowersPropagator}
and the fact that
\begin{equation}\label{Eq:HarmonicPowersPropagatorSeq}
\{ c(\alpha ) \}_{\nn d}
\mapsto
\{ e^{\zeta (2|\alpha |+d+c)^r}c(\alpha ) \}_{\nn d}
\end{equation}
is a homeomorphism on the sequence spaces which correspond
to the spaces in \eqref{Eq:Pilspaces2}.
\end{proof}

\par

We also have the following negative result concerning
continuity for the operator in
\eqref{Eq:HarmonicPowersPropagator}.

\par

\begin{prop}\label{Prop:BargPilSpacesNonCont}
Let $\zeta \in \mathbf C$ be such that
$\operatorname{Re}(\zeta )>0$, $r>0$,
$c\in \mathbf C$,
and let $s,s_1,s_2\in \overline {\mathbf R_{\flat}}$
be such that $s_1> \frac 1{2r}$ and
$s_2\ge \frac 1{2r}$. Then the following is true:
\begin{enumerate}
\item the map \eqref{Eq:HarmonicPowersPropagator}
is discontinuous from $\mascS (\rr d)$ to $\mascS '(\rr d)$;

\vrum

\item the map \eqref{Eq:HarmonicPowersPropagator}
is discontinuous from $\maclH _{0,s_1}(\rr d)$
to $\maclH _{0,s_1}'(\rr d)$;

\vrum

\item the map \eqref{Eq:HarmonicPowersPropagator}
is discontinuous from $\maclH _{s_2}(\rr d)$
to $\maclH _{s_2}'(\rr d)$.
\end{enumerate}
\end{prop}

\par

\begin{proof}
We only prove (1) and (2), and then
in the case when $s_1\in \mathbf R$ and $c=0$.
The other cases follow by similar arguments and are left
for the reader.

\par

Since $e^{itH_{x}^r}$ is a homeomorphism on
all involved spaces when $t$ is real, we may assume that
$\zeta >0$ is real.
Let
$$
f=\sum _{\alpha \in \nn d} e^{-(\log (1+|\alpha |))^2}
h_\alpha ,
$$
i.{\,}e. the Hermite coefficients for $f$ are given by
$c_h(f,\alpha )= e^{-(\log (1+|\alpha |))^2}$.
Since
$$
e^{-(\log (1+|\alpha |))^2} \lesssim \eabs \alpha ^{-N}
$$
for every $N\ge 1$, it follows that $f\in \mascS (\rr d)$
in view of \eqref{Eq:SchwSpHermChar}.

\par

On the other hand, by
\eqref{Eq:HarmonicPowersPropagator} we get
$$
c_h(e^{\zeta H_{x}^r}f,\alpha )
=
e^{\zeta (2|\alpha |+d)^r}e^{-(\log (1+|\alpha |))^2}
\gtrsim
e^{\zeta (|\alpha |+d)^r}\gtrsim \eabs \alpha ^N,
$$
for every $N\ge 1$. Hence
\eqref{Eq:TempDistHermChar} shows that
$e^{\zeta H_{x}^r}f\notin \mascS '(\rr d)$,
while $f\in \mascS (\rr d)$. This gives (1).

\par

In order to prove (2), let $s\in \mathbf R$ be
such that $\frac 1{2r}<s<s_1$, and let
$$
f=\sum _{\alpha \in \nn d} e^{-(1+|\alpha |)^{\frac 1{2s}}}
h_\alpha ,
$$
i.{\,}e. the Hermite coefficients for $f$ are given by
$c_h(f,\alpha )= e^{-(1+|\alpha |)^{\frac 1{2s}}}$.
Since
$$
e^{-(1+|\alpha |)^{\frac 1{2s}}}
\lesssim e^{-r(1+|\alpha |)^{\frac 1{2s_1}}}
$$
for every $r>0$, it follows that $f\in \maclH_{0,s_1} (\rr d)$
in view of the definitions of $\maclH_{0,s_1} (\rr d)$.

\par

On the other hand, by
\eqref{Eq:HarmonicPowersPropagator} we get
\begin{align*}
c_h(e^{\zeta H_{x}^r}f,\alpha )
&=
e^{\zeta (2|\alpha |+d)^r}e^{-(1+|\alpha |)^{\frac 1{2s}}}
\\[1ex]
&\gtrsim
e^{\zeta (|\alpha |+d)^r}
\gtrsim
e^{r(1+|\alpha |)^{\frac 1{2s}}}
\gtrsim
e^{r(1+|\alpha |)^{\frac 1{2s_1}}},
\end{align*}
for every $r>0$. Hence
$e^{\zeta H_{x}^r}f\notin \maclH _{0,s_1}'(\rr d)$,
due to the definitions,
while $f\in \maclH _{0,s_1}(\rr d)$. This gives (2),
and the result follows.
\end{proof}

\par

\subsection{Extensions to powers of $H_{x,\varrho ,c}$
and their propagators}

\par

The previous results can be extended to allow
$H_{x,\varrho ,c}$ in place of $H_{x,c}$,
for more general choices of $\varrho \in \cc d$.
We notice that 
\begin{alignat}{4}
H_{x,\varrho ,c} ^r \, &: & \ & \sum _{\alpha \in \nn 
d}c_h(f,\alpha )h_\alpha &
&\mapsto & &
\sum _{\alpha \in \nn d}(2\scal \alpha \varrho
+ \sumvec (\varrho)+c)^r
c_h(f,\alpha )h_\alpha
\tag*{(\ref{Eq:HarmonicPowers})$'$}
\intertext{and}
e^{\zeta H_{x,\varrho ,c} ^r} \, &: & \ &
\sum _{\alpha \in \nn 
d}c_h(f,\alpha )h_\alpha &
&\mapsto & & \sum _{\alpha \in \nn d}e^{
\zeta (2\scal \alpha \varrho
+ \sumvec (\varrho)+c)^r}
c_h(f,\alpha )h_\alpha
\tag*{(\ref{Eq:HarmonicPowersPropagator})$'$}
\end{alignat}
when $f\in \maclH _0'(\rr d)$, and
\begin{alignat}{3}
c&\in \mathbf C\setminus
\sets{-2\scal \alpha \varrho -\sumvec (\varrho)}
{\alpha \in \nn d} &
\quad &\text{and} &\quad r&\in \mathbf R
\tag*{(\ref{Eq:Condcr1})$'$}
\intertext{or}
c&\in \mathbf C &
\quad &\text{and} &\quad r&\in
\overline {\mathbf R}_+ ,
\tag*{(\ref{Eq:Condcr2})$'$}
\end{alignat}
Here and in what follows we let
$$
\sumvec (\varrho ) = \sum _{j=1}^d\varrho _j,
\quad \text{when}\quad
\varrho =(\varrho _1,\dots ,\varrho _d)\in \cc d.
$$

\par

By similar arguments as in the proof of
Proposition \ref{Prop:BargPilSpaces} we
get the following extension. The details are
left for the reader.

\par

\renewcommand{\rubrik}{Proposition
\ref{Prop:BargPilSpaces}$'\!\!$}

\par

\begin{tom}
Let $\zeta \in \mathbf C$,$r\in \mathbf R$
$\varrho \in \cc d$ and $c\in \mathbf C$ be as in
\eqref{Eq:Condcr1}$'$ or as in \eqref{Eq:Condcr2}$'$,
and let $s,s_1,s_2\in \overline {\mathbf R_{\flat}}$
be such that $0<s_1\le \frac 1{2r}$ and
$s_2< \frac 1{2r}$. Then the following is true:
\begin{enumerate}
\item the map \eqref{Eq:HarmonicPowers}$'$
on $\maclH _0'(\rr d)$ restricts
to continuous mappings on $\mascS (\rr d)$,
$\mascS '(\rr d)$
and on the spaces in \eqref{Eq:Pilspaces}.
If in addition \eqref{Eq:Condcr1}$'$ holds true,
then these mappings are homeomorphisms;

\vrum

\item the map \eqref{Eq:HarmonicPowersPropagator}$'$ on
$\maclH _0'(\rr d)$ restricts
to homeomorphisms on
\begin{equation}\tag*{(\ref{Eq:Pilspaces2})$'$}
\maclH _{0,s_1}(\rr d),\quad \maclH _{s_2}(\rr d),
\quad \maclH _{s_2}'(\rr d)
\quad \text{and on}\quad
\maclH _{0,s_1}'(\rr d).
\end{equation}
If in addition $\zeta \varrho _j^r \in i\mathbf R$
for every $j$, then the map
\eqref{Eq:HarmonicPowersPropagator}$'$ is homeomorphic
on $\mascS (\rr d)$, $\mascS '(\rr d)$ and the spaces
in \eqref{Eq:Pilspaces}
\end{enumerate}
\end{tom}

\par




\par

In the same way, similar arguments as in
the proof of Proposition
\ref{Prop:BargPilSpacesNonCont} gives
the following extension. The details are left
for the reader.

\par

\renewcommand{\rubrik}{Proposition 
\ref{Prop:BargPilSpacesNonCont}$'$}

\par

\begin{tom}
Let $r>0$, $\varrho \in \cc d$ be such that
$\operatorname{Re}(\zeta \varrho _j^r )>0$ for
every $j=1,\dots ,d$,
$c\in \mathbf C$,
and let $s,s_1,s_2\in \overline {\mathbf R_{\flat}}$
be such that $s_1> \frac 1{2r}$ and
$s_2\ge \frac 1{2r}$. Then the following is true:
\begin{enumerate}
\item the map \eqref{Eq:HarmonicPowersPropagator}
is discontinuous from $\mascS (\rr d)$ to $\mascS '(\rr d)$;

\vrum

\item the map \eqref{Eq:HarmonicPowersPropagator}
is discontinuous from $\maclH _{0,s_1}(\rr d)$
to $\maclH _{0,s_1}'(\rr d)$;

\vrum

\item the map \eqref{Eq:HarmonicPowersPropagator}
is discontinuous from $\maclH _{s_2}(\rr d)$
to $\maclH _{s_2}'(\rr d)$.
\end{enumerate}
\end{tom}

\par

\subsection{Powers of harmonic oscillator
on Hilbert modulations spaces}\label{subsec2.3}

\par

Let $r\in \mathbf R$, $p,q\in (0,\infty ]$, $s>1$,
$\omega \in \mascP _{E,s}(\rr {2d})$ and
$\vartheta _r(x,\xi )=\eabs {(x,\xi )}^r$. Then
it follows from \cite[Proposition 1.34$'$]{AbCoTo}
that $H_x^N$ is homeomorphic from
$M^{p,q}_{(\omega \vartheta _N)}(\rr d)$ to
$M^{p,q}_{(\omega )}(\rr d)$ when $N$ is an integer. So far,
we are not able to prove any
such result when the integer $N$ is replaced by
a general real number for such general modulation spaces.
On the other hand we have the following
for certain Hilbert modulation spaces. See also
\cite{BogCorGro,Dau} for similar results. Here we
restrict ourself to weights which are rotational
invariant with respect to each phase space, or complex
variable. That is, we assume that
\begin{equation}\label{Eq:RotInvWeights}
\omega (x,\xi ) = \omega _0(\rho ),
\qquad
\rho _j = x_j^2+\xi _j^2,
\end{equation}
for some positive function $\omega _0$
on ${\overline {\mathbf R}}_+^d$.

\par

\begin{thm}\label{Thm:FracHarmOscModSp}
Let $r,t\in \mathbf R$, $\vartheta _r(x,\xi )=\eabs {(x,\xi )}^r$
and suppose that $\omega \in \mascP _E(\rr {2d})$ satisfies
\eqref{Eq:RotInvWeights} for some positive function $\omega _0$
on ${\overline {\mathbf R}}_+^d$.
Then the following is true:
\begin{enumerate}
\item $H_x^{r}$ on $\maclH _{0}(\rr d)$ is uniquely extendable to
a homeomorphism from $M^{2}_{(\omega )}(\rr d)$ to
$M^{2}_{(\omega /\vartheta _{r})}(\rr d)$, with bound of $H_x^{r}$
which is independent of $r$;

\vrum

\item $e^{itH_x^{r}}$ on $\maclH _{0}(\rr d)$ is uniquely extendable to
a homeomorphism on $M^{2}_{(\omega )}(\rr d)$, with bound
of $e^{itH_x^{r}}$ which is independent of $t$ and $r$.
\end{enumerate}
\end{thm}

\par

The following lemma shows that $M^{2}_{(\omega )}(\rr d)$
with weights of the form  \eqref{Eq:RotInvWeights} can be
expressed as norm estimates of Hermite
series expansions of the involved functions and distributions.
Here we extract the weight function
$$
\nu _\omega : \nn d \to \mathbf R_+,
$$
from $\omega$ by the formula
\begin{equation}\label{Eq:DiscWeightDef}
\begin{aligned}
\nu _\omega (\alpha )
&\equiv
\left (
\alpha !^{-1}
\int _{\rr d_+}r^\alpha \omega _0(r)^2e^{-(r_1+\cdots +r_d)}
\, dr \right )^{\frac 12},\qquad 
\alpha \in \nn d.
\end{aligned}
\end{equation}

\par

\begin{lemma}\label{Lemma:FracHarmOscModSp}
Let $\omega \in \mascP _E(\rr {2d})$ be such that
\eqref{Eq:RotInvWeights} holds for some positive function $\omega _0$
on ${\overline {\mathbf R}}_+^d$. Also let $\nu _\omega$ be given by
\eqref{Eq:DiscWeightDef}.
Then $M^{2}_{(\omega )}(\rr d)$
consists of all $f\in \mascS '(\rr d)$ such that
\begin{equation}\label{Eq:FracHarmOscModSp}
\nm f{[\omega ]}
\equiv
\left ( \sum _{\alpha \in \nn d} |c_h(f,\alpha )\nu _\omega (\alpha )|^2\right )^{\frac 12}
\end{equation}
is finite. The Hilbert norm $\nm \cdo{[\omega ]}$ is equivalent to
$\nm \cdo {M^{2}_{(\omega )}}$.
\end{lemma}

\par

\begin{proof}
The result follows by straight-forward applications of
\cite[Theorem 3.5]{Toft18}, and that the Bargmann image
of $M^2_{(\omega)}(\rr d)$ is equal to $A^2_{(\omega )}(\cc d)$.
The details are left for the reader.
\end{proof}

\par

\begin{proof}[Proof of Theorem \ref{Thm:FracHarmOscModSp}]
Let $\nm \cdo{[\omega ]}$ be the norm given in \eqref{Eq:FracHarmOscModSp}.
Then \eqref{Eq:HarmonicPowers} shows that
\begin{align*}
c_h(H_x^r f,\alpha )&=(2|\alpha |+d)^rc(f,\alpha )
\intertext{and}
c_h(e^{itH_x^r} f,\alpha )&=e^{it(2|\alpha |+d)^r}c(f,\alpha ).
\end{align*}
Since $|e^{it(2|\alpha |+d)^r}|=1$ and that $\omega$ satisfies
\eqref{Eq:RotInvWeights}, the assertion follows from Lemma
\ref{Lemma:FracHarmOscModSp}.
\end{proof}

\par

\begin{rem}
So far we are not able to extend
Theorem \ref{Thm:FracHarmOscModSp} to more general
modulation spaces $M^{p,q}_{(\omega)}(\rr d)$, without
the assumption \eqref{Eq:RotInvWeights} on
the weight $\omega$. Suppose that $\omega$
is the same as in Theorem \ref{Thm:FracHarmOscModSp}.
Then using standard embeddings for modulation spaces like
\begin{equation}\label{Eq:ModEmb}
\begin{aligned}
M^{p,q}_{(\omega \vartheta _{\tau _1})}(\rr d)
&\hookrightarrow
M^{2}_{(\omega )}(\rr d )
\hookrightarrow
M^{p,q}_{(\omega /\vartheta _{\tau _2})}(\rr d),
\\[1ex]
\tau _1 &\ge d \max \Big (0, \frac 12-\frac 1p,\frac 12-\frac 1q \Big ),
\\[1ex]
\tau _2 &\ge d \max \Big (0, \frac 1p-\frac 12,\frac 1q-\frac 12 \Big ),
\end{aligned}
\end{equation}
with strict inequalities if $p\neq 2$ or $q\neq 2$, it follows that
\begin{equation}
H^{r_0}e^{itH^r} : M^{p,q}_{(\omega )}(\rr d)
\to
M^{p,q}_{(\omega /\vartheta _\theta )}(\rr d),
\quad
\theta = r_0+\tau _1+\tau _2
\end{equation}
is continuous, with bounds which are independent
of $r$. In fact, by
Theorem \ref{Thm:FracHarmOscModSp} and
\eqref{Eq:ModEmb} we have
\begin{align*}
H^{r_0}e^{itH^r}M^{p,q}_{(\omega )}(\rr d)
&\hookrightarrow
H^{r_0}e^{itH^r}M^{2}_{(\omega /\vartheta _{\tau _1})}(\rr d)
\\[1ex]
&\hookrightarrow
M^{2}_{(\omega /\vartheta _{r_0+\tau _1})}(\rr d)
\hookrightarrow
M^{p,q}_{(\omega /\vartheta _{\theta})}(\rr d).
\end{align*}

\end{rem}

\par

\section{General harmonic oscillator propagators
and fractional Fourier transforms}\label{sec3}

\par

In this section we prove that generalized
harmonic oscillator propagators of
the forms $e^{irH_{x,\varrho ,c}}$ are essentially 
fractional Fourier transforms of multiple orders.
We use such identities to link some results in \cite{Bhi}
with results in \cite{Toft18}, especially when such 
operators act on (weighted) modulation spaces.

\par

\subsection{Identifications between
fractional Fourier transforms and harmonic
oscillator type propagators}

\par

Let $\varrho \in \cc d$ and $c\in \mathbf C$.
Then the operator
$H_{x,\varrho ,c}$ is transformed into the operator
\begin{align}
2H _{\varrho ,c,\mathfrak V} &=
2H_{\varrho ,\mathfrak V} +\sumvec (\varrho )+c,
\quad \text{where}\quad
H_{\varrho ,\mathfrak V} 
= \sum _{j=1}^d \varrho _j z_j\partial _{z_j}.
\notag
\intertext{That is,}
\mathfrak V_d \circ H_{x,\varrho ,c}
&= (2H_{\varrho ,\mathfrak V} +\sumvec (\varrho )+c)\circ \mathfrak V_d,
\quad
H_{\varrho ,\mathfrak V} 
= \sum _{j=1}^d \varrho _j z_j\partial _{z_j},
\label{Eq:HarmOscBargmTransf}
\intertext{or equivalently,}
\mathfrak V_d H_{x,\varrho ,c}f 
&= (2H_{\varrho ,\mathfrak V} +\sumvec (\varrho )+c)F,
\qquad
F=\mathfrak V_df.
\notag
\end{align}
For convenience we put
$$
H _{\varrho _0,c,\mathfrak V} 
= 
H _{\varrho ,c,\mathfrak V}
\quad \text{and}\quad 
H _{\varrho _0,\mathfrak V} 
= 
H _{\varrho ,\mathfrak V},
$$
when $\varrho =(\varrho _0,\dots ,\varrho _0)\in \cc d$,
and we put
$$
H _{\mathfrak V}=H _{1,\mathfrak V}.
$$
In particular, it follows from 
\eqref{Eq:HarmOscBargmTransf} that
$\frac 12H_{x,-d}$ is transformed into 
$H_{\mathfrak V}$. It also follows from
\eqref{Eq:HarmOscBargmTransf} that if $r,\zeta \in \mathbf C$
satisfy $\RE(r)\ge 0$, then
\begin{equation}
\label{Eq:HarmOscExpBargmTransf}
\mathfrak V_d \circ e^{\zeta H_{x,\varrho ,c}^r}
=
e^{\zeta (2H_{\varrho ,\mathfrak V} +\sumvec (\varrho )+c)^r}
\circ \mathfrak V_d,
\end{equation}
as continuous operators from $\maclH _0'(\rr d)$ 
to $\maclA _0'(\cc d)$.

\par

By straight-forward computations we get
$$
H_{\varrho ,\mathfrak V}e_\alpha (z) 
= \scal \varrho \alpha e_\alpha (z)
$$
(see e.{\,}g. \cite{Ba1}). This implies that
\begin{align}
e^{\zeta H_{\varrho ,\mathfrak V}}e_\alpha (z)
&=
e^{\zeta \scal \varrho \alpha }e_{\alpha}(z)
=
e_{\alpha}(e^{\zeta \varrho _1}z_1,\dots ,
e^{\zeta \varrho _d}z_d) ,
\label{Eq:HarmPropBargmannPrel}
\intertext{for every $\zeta \in \mathbf C$,
which gives}
(e^{\zeta H_{\varrho ,\mathfrak V}}F)(z)
&=
F(e^{\zeta \varrho _1}z_1,\dots ,
e^{\zeta \varrho _d}z_d) ,
\qquad F\in \maclA _0'(\cc d).
\label{Eq:HarmPropBargmann}
\end{align}

\par

\begin{rem}
We observe that the map
$$
F(z)\mapsto (e^{\zeta H_{\varrho ,\mathfrak V}}F)(z)
=
F(e^{\zeta \varrho _1}z_1,\dots ,
e^{\zeta \varrho _d}z_d)
$$
is a continuous bijection on the spaces
\begin{equation}\label{Eq:BargPilSpaces}
\maclA _{s_1}(\cc d), \quad \maclA _{0,s_2}(\cc d),
\quad
\maclA _{s_1}'(\cc d), \quad \maclA _{0,s_2}'(\cc d),
\end{equation}
when $s_1<\frac 12$ and $s_2\le \frac 12$. See 
\cite{Toft18}
for definition and some characterizations of the spaces
in \eqref{Eq:BargPilSpaces}. This is also a consequence of
Proposition \ref{Prop:BargPilSpaces}.
\end{rem}

\par

We recall that the fractional Fourier transform 
$\mascF _{\! \varrho}$ of (multiple) order
$\varrho \in \cc d$ satisfies
\begin{equation}\label{Eq:BargmannActFracFourTr}
(\mathfrak V_d(\mascF _{\! \varrho} f))(z) = (\mathfrak 
V_df)(e^{-i\frac {\pi \varrho _1}2}z_1,\dots ,
e^{-i\frac {\pi \varrho _d}2}z_d),
\qquad f\in \maclH _0'(\rr d).
\end{equation}
Hence, a combination of Proposition 
\ref{Prop:BargPilSpaces}$'$,
\eqref{Eq:HarmOscBargmTransf},
\eqref{Eq:HarmOscExpBargmTransf}
and \eqref{Eq:HarmPropBargmann} with
$$
e^{-iH_{x,\varrho ,c_1}}
=
e^{i(c_2-c_1)}e^{-iH_{x,\varrho ,c_2}} ,
\qquad c_1,c_2\in \mathbf C
$$
gives the following extention of results given in
\cite[p. 161]{KutOza}.

\par

\begin{thm}\label{Thm:HarmPropFracFourT}
Let $\varrho \in \cc d$, $c\in \mathbf C$ and $s,s_1,s_2
\in \overline {\mathbf R_{\flat}}$ be such that $0<s_1\le \frac 1{2}$
and $s_2< \frac 1{2}$. Then
\begin{equation}\label{Eq:FracFourTransHarmProp}
e^{-i\frac {\pi}4 H_{x,\varrho ,c}}
= e^{-i\frac \pi 4(\sumvec (\varrho )+c)} \mascF _{\! \varrho}
\end{equation}
as operators on $\maclH _0'(\rr d)$. The operators in
\eqref{Eq:FracFourTransHarmProp} restrict to
homeomorphisms on the spaces in \eqref{Eq:Pilspaces2}.
\end{thm}

\par

Evidently, \eqref{Eq:FracFourTransHarmProp}
is the same as
\begin{equation}\tag*{(\ref{Eq:FracFourTransHarmProp})$'$}
\mascF _{\! \varrho}
=
e^{i\frac \pi 4(\sumvec (\varrho )+c)}
e^{-i\frac {\pi}4 H_{x,\varrho ,c}},
\end{equation}
and can be used to transfer
properties between fractional Fourier transform and harmonic
oscillator propagators.

\par

By combining    with Theorem \ref{Thm:HarmPropFracFourT},
we get the following extensions of Propositions
\ref{Prop:IntroFracFourComplOrder} and
\ref{Prop:IntroFracFourComplOrderRefine} from the introduction.
The details are left for the reader.

\par

\renewcommand{\rubrik}{Proposition
\ref{Prop:IntroFracFourComplOrder}$'\!\!$}

\par

\begin{tom}
Let $I_d=\{ 1,\dots ,d\}$, $\varrho \in \cc d$ and
$s\in \overline{\mathbf R}_\flat$.
Then the following is true:
\begin{enumerate}
\item if $s< \frac 12$, then $\mascF _{\! \varrho}$
and $e^{-i\frac \pi 4H_{x,\varrho}}$ are homeomorphisms on
$\maclH _s(\rr d)$ and on $\maclH _s'(\rr d)$;

\vrum

\item if $\IM (\varrho _j)\le 0$ for every $j\in I_d$,
$\IM (\varrho _{j_0})< 0$ for some $j_0\in I_d$
and $s\ge \frac 12$, then
$\mascF _{\! \varrho}$
and $e^{-i\frac \pi 4H_{x,\varrho}}$ are continuous
injections but not surjections on
$\maclH _s(\rr d)$, $\mascS (\rr d)$, $\mascS '(\rr d)$
and on $\maclH _s'(\rr d)$;

\vrum

\item if $\IM (\varrho _j)=0$ for every $j\in I_d$, then
$\mascF _{\! \varrho}$
and $e^{-i\frac \pi 4H_{x,\varrho}}$ are homeomorphisms
on $\maclH _s(\rr d)$, $\mascS (\rr d)$, $\mascS '(\rr d)$
and on $\maclH _s'(\rr d)$;

\vrum

\item if $\IM (\varrho _j)>0$ for some $j\in I_d$ and
$s\ge \frac 12$, then $\mascF _{\! \varrho}$
and $e^{-i\frac \pi 4H_{x,\varrho}}$ are discontinuous on
$\maclH _s(\rr d)$, $\mascS (\rr d)$, $\mascS '(\rr d)$
and on $\maclH _s'(\rr d)$.
\end{enumerate}
The same holds true with $s> \frac 12$, $s\le \frac 12$
and $\maclH _{0,s}$ in place of $s\ge \frac 12$, $s< \frac 12$
and $\maclH _s$ at each occurrence.
\end{tom}

\par

\renewcommand{\rubrik}{Proposition
\ref{Prop:IntroFracFourComplOrderRefine}$'\!\!$}

\par

\begin{tom}
Let $I_d=\{ 1,\dots ,d\}$ and $\varrho \in \cc d$. Then the following is true:
\begin{enumerate}
\item if $\IM (\varrho _j)<0$ for every $j\in I_d$, then $\mascF _{\! \varrho}$
and $e^{-i\frac \pi 4H_{x,\varrho}}$ are continuous from
$\maclS _{1/2}'(\rr d)$ to $\maclS _{1/2}(\rr d)$, and
$$
\mascF _{\! \varrho} (\maclS _{1/2}'(\rr d)) =
e^{-i\frac \pi 4H_{x,\varrho}} (\maclS _{1/2}'(\rr d)) \subsetneq \maclS _{1/2}(\rr d)
\text ;
$$

\item if $\IM (\varrho _j)>0$ for every $j\in I_d$,
then $\mascF _{\! \varrho}$
and $e^{-i\frac \pi 4H_{x,\varrho}}$ are discontinuous from
$\maclS _{1/2}(\rr d)$ to $\maclS _{1/2}'(\rr d)$, and
$$
\maclS _{1/2}'(\rr d) \subsetneq
\mascF _{\! \varrho} (\maclS _{1/2}(\rr d))
=
e^{-i\frac \pi 4H_{x,\varrho}}(\maclS _{1/2}(\rr d))
\subsetneq
\maclH _{0,1/2}'(\rr d)\text ;
$$

\item if $\IM (\varrho _j)>0$ for some $j\in I_d$,
then $\mascF _{\! \varrho}$
and $e^{-i\frac \pi 4H_{x,\varrho}}$ are discontinuous from
$\maclS _{1/2}(\rr d)$ to $\maclS _{1/2}'(\rr d)$, and
$$
\mascF _{\! \varrho} f \in \maclH _{0,1/2}'(\rr d)\setminus \maclS _{1/2}'(\rr d)
\quad \text{and}\quad
e^{-i\frac \pi 4H_{x,\varrho}} f \in
\maclH _{0,1/2}'(\rr d)\setminus \maclS _{1/2}'(\rr d)
$$
for some $f\in \maclS _{1/2}(\rr d)$.
\end{enumerate}
\end{tom}

\par

In the following proposition we point out some auxiliary
group properties for fractional Fourier transforms of complex orders,
which extends similar results in \cite{KutOza} for fractional
Fourier transforms of real orders. The result follows by
straight-forward applications of
\eqref{Eq:BargmannActFracFourTr} and the fact that the Bargmann
transform is injective. The details are left for the reader.

\par

%
%
%
%
%
%

\begin{prop}
For any $\varrho \in \cc d$, let $\mascF _{\! \varrho}$ be acting on
$\maclH _0'(\rr d)$. Then $\{ \mascF _{\! \varrho} \} _{\varrho \in \cc d}$
is a commutative
group under composition, with identity element $\mascF _{0}
=\operatorname{Id}_{\maclH _0'(\rr d)}$, and
\begin{equation}
\begin{aligned}
\mascF _{\! \varrho _1}\circ \mascF _{\! \varrho _2}
= \mascF _{\! \varrho _1+\varrho _2}, \qquad
\mascF _{\! \varrho}^{-1} = \mascF _{\! -\varrho},
\qquad
\mascF _{\! \varrho + \varrho _0}
&= \mascF _{\! \varrho},
\\[1ex]
\varrho ,\varrho _1,\varrho _2 &\in \cc d, \ 
\varrho _0\in 4\zz d.
\end{aligned}
\end{equation}
\end{prop}

\par

\subsection{Continuity for one-parameters
fractional Fourier transforms and
harmonic oscillator propagators on modulation spaces}
As an example we shall next transfer
mapping properties of $\mascF _{\! \varrho}$,
$\varrho \in \mathbf R$,
when acting on certain classes
of modulation spaces 
into analogous properties
for harmonic oscillator propagators.

\par

For fractional Fourier transforms on modulation
spaces we recall the following special case of 
\cite[Proposition 7.1]{Toft18}. We refer to
Subsection \ref{subsec1.4} for notations on
weight classes.

\par

\begin{prop}\label{Prop:FracFourMod}
Let $\varrho \in \mathbf R$, $p\in (0,\infty ]$
and $\omega \in \mascP _{\! A,r}(\rr {2d})$.
Then $\mascF _{\! \varrho}$ is an isometric homeomorphism on $M^p_{(\omega )}(\rr d)$.
\end{prop}

\par

A combination of Theorem \ref{Thm:HarmPropFracFourT}
and Proposition \ref{Prop:FracFourMod} now gives the
following. (See \cite[Theorem 2.6]{Bhi} and
\cite[Theorem 1.7]{BhiBalTha}
in the case when $p\ge1$ and $\omega =1$. See also
\cite{CorGroNicRod} in the case when $p\ge 1$ and $\omega$
is moderatated by polynomially bounded weights.)

\par

\begin{prop}\label{Prop:HarmPropMod}
Let $c\in \mathbf C$, $\varrho \in \mathbf R$,
$p\in (0,\infty ]$ and $\omega \in
\mascP _{\! A,r}(\rr {2d})$.
Then $e^{-i\frac \pi 4 H_{x,\varrho ,c}}$ is a
homeomorphism on $M^p_{(\omega )}(\rr d)$.
\end{prop}

\par

We also have the following extensions of
\cite[Proposition 5.1]{CorNic}
and the previous propositions.
Here the involved weights are allowed to belong to
the general class
$\mascP _{\! A}(\rr {2d})$, and are linked as
\begin{equation}\label{Eq:FracFourWeightsRel}
\begin{aligned}
\omega _\varrho (x,\xi )
&=
\omega (A_{d,\varrho}(x,\xi ))
\end{aligned}
\end{equation}
(see Subsections \ref{subsec1.4}, \ref{subsec1.7}
and Remark \ref{Rem:GeneralModSpaces})

%
%


\par

\begin{thm}\label{Thm:FrFTModSpec1}
Let $\varrho \in \mathbf R\setminus 2\mathbf Z$,
$c\in \mathbf C$,
$\omega ,\omega _\varrho \in
\mascP _{\! A}(\rr {2d})$
and $p,q\in (0,\infty]$ be such that $q\le p$ and
\eqref{Eq:FracFourWeightsRel} hold. Then
$\mascF _{\! \varrho}$ and
$e^{-i\frac \pi 4H_{x,\varrho ,c}}$
on $\maclH _{\flat _1}'(\rr d)$
restrict to continuous mappings from
$M^{p,q}_{(\omega )}(\rr d)$
to $M^{q,p}_{(\omega _\varrho )}(\rr d)$, and from
$W^{q,p}_{(\omega )}(\rr d)$
to $W^{p,q}_{(\omega _\varrho )}(\rr d)$, and
\begin{alignat}{2}
\nm {\mascF _{\! \varrho} f}
{M^{q,p}_{(\omega _\varrho )}}
&=
e^{-\frac \pi 4 \cdot \IM(c)}
\nm {e^{-i\frac \pi 4 H_{x,\varrho ,c}}f}
{M^{q,p}_{(\omega _\varrho )}}
\notag
\\
&\lesssim
|\sin (\textstyle{\frac {\pi \varrho}2})|^{d
(\frac 1p-\frac 1q)}
\nm f{M^{p,q}_{(\omega )}}, &
\qquad f&\in M^{p,q}_{(\omega )}(\rr d)
\label{Eq:FrFTModSpec1A}
\intertext{and}
\nm {\mascF _{\! \varrho} f}{W^{p,q}_{(\omega _\varrho )}}
&=
e^{-\frac \pi 4 \cdot \IM (c)}
\nm {e^{-i\frac \pi 4H_{x,\varrho ,c}}f}
{W^{p,q}_{(\omega _t)}}
\notag
\\
&\lesssim
|\sin (\textstyle{\frac {\pi \varrho }2})|
^{d(\frac 1p-\frac 1q)}
\nm f{W^{q,p}_{(\omega )}}, &
\qquad f&\in W^{q,p}_{(\omega )}(\rr d) .
\label{Eq:FrFTModSpec1B}
\end{alignat}
\end{thm}

\par

\begin{thm}\label{Thm:FrFTModSpec2}
Let $\varrho \in \mathbf R\setminus (2\mathbf Z +1)$,
$\omega ,\omega _\varrho \in \mascP _{\! A}(\rr {2d})$
and $p,q\in (0,\infty]$ be such that $q\le p$ and
\eqref{Eq:FracFourWeightsRel} hold.
Then $\mascF _{\! \varrho} $ and
$e^{-i\frac \pi 4H_{x,\varrho ,c}}$
on $\maclH _{\flat _1}'(\rr d)$
restrict to continuous mappings from
$M^{p,q}_{(\omega )}(\rr d)$
to $W^{p,q}_{(\omega _\varrho )}(\rr d)$, and from
$W^{q,p}_{(\omega )}(\rr d)$
to $M^{q,p}_{(\omega _\varrho )}(\rr d)$, and
\begin{alignat}{2}
\nm {\mascF _{\! \varrho} f}{W^{p,q}_{(\omega _\varrho )}}
&=
e^{-\frac \pi 4 \cdot \IM(c)}
\nm {e^{-i\frac \pi 4 H_{x,\varrho ,c}}f}
{W^{p,q}_{(\omega _\varrho )}}
\notag
\\
&\lesssim
|\cos (\textstyle{\frac {\pi \varrho }2})|
^{d(\frac 1p-\frac 1q)}
\nm f{M^{p,q}_{(\omega )}}, &
\qquad f&\in M^{p,q}_{(\omega )}(\rr d)
\label{Eq:FrFTModSpec2A}
\intertext{and}
\nm {\mascF _{\! \varrho} f}{M^{q,p}_{(\omega _\varrho )}}
&=
e^{-\frac \pi 4 \cdot \IM (c)}
\nm {e^{-i\frac \pi 4H_{x,\varrho ,c}}f}
{M^{q,p}_{(\omega _t)}}
\notag
\\
&\lesssim
|\cos (\textstyle{\frac {\pi \varrho }2})|
^{d(\frac 1p-\frac 1q)}
\nm f{W^{q,p}_{(\omega )}}, &
\qquad f&\in W^{q,p}_{(\omega )}(\rr d).
\label{Eq:FrFTModSpec2B}
\end{alignat}
\end{thm}

\par

For the proofs of Theorems \ref{Thm:FrFTModSpec1} and
\ref{Thm:FrFTModSpec2} we need the following version of Minkowski's
inequality.

\par

%
%

\begin{lemma}\label{Lemma:MinkowskiSpec}
Let $\varrho \in \mathbf R$, $p,q\in (0,\infty ]$
be such that $q\le p$,
$A_{d,\varrho }$ be as in Subsection \ref{subsec1.7}
and let $T_\varrho$ from $\Sigma _1(\rr {2d})$
to $\Sigma _1(\rr {2d})$ be given by
$$
(T_\varrho f)(x,\xi ) = f(A_{d,\varrho}(x,\xi )),
\qquad f\in L^q_{loc}(\rr {2d}).
$$
Then the following is true:
\begin{enumerate}
\item if $\varrho \neq 2\mathbf Z$, then
$T_\varrho$ extends uniquely to a
continuous map from
$L^{p,q}(\rr {2d})$ to $L^{q,p}(\rr {2d})$
and from $L^{q,p}_*(\rr {2d})$
$L^{p,q}_*(\rr {2d})$, and
\begin{alignat}{2}
\nm {T_\varrho f}{L^{q,p}}
&\le
|\sin {\textstyle{(\frac {\pi \varrho }2)}}|
^{d(\frac 1p-\frac 1q)} \nm f{L^{p,q}}, &
\qquad f&\in L^q_{loc}(\rr {2d}) \text ;
\label{Eq:MinkowskiSpec1A}
\end{alignat}

\vrum

\item if $\varrho \neq 2\mathbf Z+1$, then
$T_\varrho$ extends uniquely to a
continuous map from $L^{p,q}(\rr {2d})$ to $L^{p,q}_*(\rr {2d})$
and from $L^{q,p}_*(\rr {2d})$ to $L^{q,p}(\rr {2d})$, and
\begin{alignat}{2}
\nm {T_\varrho f}{L^{p,q}_*}
&\le
|\cos {\textstyle{(\frac {\pi \varrho }2)}}|
^{d(\frac 1p-\frac 1q)} \nm f{L^{p,q}}, &
\qquad f&\in L^q_{loc}(\rr {2d}) \text .
\label{Eq:MinkowskiSpec2A}
\end{alignat}
The same holds true with $L^{p,q}_*$ and
$L^{q,p}_*$ in place of
$L^{q,p}$ and $L^{p,q}$, respectively,
at each occurrence.
\end{enumerate}
\end{lemma}

\par

\begin{proof}
We only prove (1). The assertion (2) follows by similar 
arguments and is left for the reader.

\par

First suppose that $p<\infty$. Then $q<\infty$.
Let $\theta =\frac {\pi \varrho }2$,
$f\in \mascS (\rr d)$, $p_0=p/q\ge 1$,
$h\in \mascS (\rr d)$ be such
that $\nm h{L^{p_0'}}\le 1$,
$$
f_\varrho (x,\xi ) = f(A_{d,\varrho}(x,\xi ))
\quad \text{and}\quad
g_\varrho (\xi ) = \int _{\rr d}|f_\varrho (x,\xi )|^q\, dx.
$$
Since $f_0=f$ and
$$
(y,\eta ) = A_{d,\varrho}(x,\xi )
\quad \Leftrightarrow \quad
(x,\xi ) = A_{d,-\varrho}(y,\eta ),
$$
we get
\begin{align*}
|(g_\varrho ,h)_{L^2}|
&=
\left |
\iint _{\rr {2d}} |f_0(A_{d,\varrho}
(x,\xi ))|^qh(\xi )\, dxd\xi
\right |
\\[1ex]
&=
\left |
\iint _{\rr {2d}} |f_0(y,\eta )|^qh((\sin \theta)y+(\cos \theta)\eta )\, dxd\xi
\right |
\\[1ex]
&\le
\int _{\rr {d}} \nm {|f_0(\cdo ,\eta )|^q}{L^{p_0}}
\nm {h((\sin \theta)\cdo +(\cos \theta)\eta )}{L^{p_0'}}\, d\xi
\\[1ex]
&=
|\sin \theta |^{-\frac d{p_0'}}\nm {h}{L^{p_0'}}
\int _{\rr {d}} \nm {|f_0(\cdo ,\eta )|^q}{L^{p_0}}
\, d\xi
\le
|\sin \theta|^{-\frac d{p_0'}} \nm {f_0}{L^{p,q}}^q.
\end{align*}
By taking the supremum over all possible $h$ with $\nm h{L^{p_0'}}\le 1$
we obtain
$$
\nm {g_\varrho}{L^{p_0}}
\le
|\sin \theta|^{-\frac d{p_0'}} \nm {f_0}{L^{p,q}}^q
=
|\sin \theta|^{-\frac d{p_0'}} \nm {f}{L^{p,q}}^q,
$$
which is the same as \eqref{Eq:MinkowskiSpec1A}
when $f\in \mascS (\rr {2d})$.
Since $\mascS (\rr {2d})$ is dense in
$L^{p,q}(\rr {2d})$ when $p,q<\infty$,
\eqref{Eq:MinkowskiSpec1A} follows when $p<\infty$.

\par

Next suppose that $p=\infty$. The result is
obviously true when $q=\infty$. Therefore
suppose that $q<\infty$, and let
$f\in L^q_{loc}(\rr {2d})$ and
$$
g(\xi ) = \nm {f(\cdo ,\xi )}{L^p}.
$$
Then
$$
\int _{\rr d}|f_\varrho (x,\xi )|^q\, dx
\le
\int _{\rr d}|g((-\sin \theta )x+(\cos \theta )\xi )|^q\, dx
=|\sin \theta |^{-d}\nm g{L^{\infty ,q}}^q
$$
If we take the supremum over $\xi \in \rr d$, then
we obtain \eqref{Eq:MinkowskiSpec1A} for $p=\infty$,
and we have proved \eqref{Eq:MinkowskiSpec1A} for
any $p\in (0,\infty ]$.
%
%
This gives (1), and the result follows.
\end{proof}

\par

\begin{proof}[Proof of Theorems \ref{Thm:FrFTModSpec1} and
\ref{Thm:FrFTModSpec2}]
By
$$
\mascF (M^{p,q}_{(\omega _1)}(\rr d)) = W^{q,p}_{(\omega _2)}(\rr d),
\quad \text{when}\quad
\omega _1(x,\xi ) = \omega _2(\xi ,-x)
$$
and that
$|\cos (\frac {\pi (\varrho +1)}2)|
=|\sin (\frac {\pi \varrho }2)|$,
it follows that
Theorem \ref{Thm:FrFTModSpec2} is the Fourier version of
Theorem \ref{Thm:FrFTModSpec1}. Hence it suffices to prove
Theorem \ref{Thm:FrFTModSpec1}.
By \eqref{Eq:FracFourTransHarmProp} it also follows
that it suffices to prove the norm estimates for
$\mascF _{\! \varrho} f$ in
\eqref{Eq:FrFTModSpec1A} and \eqref{Eq:FrFTModSpec1B}.

\par

%

Let $\phi (x)=\pi ^{-\frac d4}e^{-\frac 12|x|^2}$. Then
\eqref{Eq:STFTFracFour} gives
\begin{align*}
F_\varrho (x,\xi )
&=
|(V_\phi (\mascF _{\! \varrho} f))(x,\xi )
\omega _\varrho (x,\xi )|
=
|(V_\phi f)(A_{d,\varrho}(x,\xi ))
\omega (A_{d,\varrho}(x,\xi ))|.
\intertext{We also have}
F_0(x,\xi )
&=
|(V_\phi f)(x,\xi )\omega (x,\xi )| 
\quad \text{and}\quad
F_\varrho (x,\xi )
=
F_0(A_{d,\varrho}(x,\xi )).
\end{align*}
By Lemma \ref{Lemma:MinkowskiSpec} and the identities
above we obtain
$$
\nm {\mascF _{\! \varrho} f}{M^{q,p}_{(\omega _\varrho )}}
\asymp
\nm {F_\varrho}{L^{q,p}}
\le
|\sin {\textstyle{(\frac {\pi \varrho }2)}}|
^{d(\frac 1p-\frac 1q)} \nm {F_0}{L^{p,q}}
\asymp
|\sin {\textstyle{(\frac {\pi \varrho }2)}}|
^{d(\frac 1p-\frac 1q)}
\nm {f}{M^{p,q}_{(\omega )}},
$$
and \eqref{Eq:FrFTModSpec1A} follows. In similar ways one
obtains \eqref{Eq:FrFTModSpec1B}. The details are left
for the reader, and the result follows.
\end{proof}

\par

\begin{rem}
By choosing $\omega =1$ and $p=q'\ge 2$, Theorem
\ref{Thm:FrFTModSpec2} agrees with
\cite[Proposition 5.1]{CorNic} by Cordero and Nicola.
\end{rem}

\par

It is evident that Theorems \ref{Thm:FrFTModSpec1} and
\ref{Thm:FrFTModSpec2} implies the following weighted
version of Proposition \ref{Prop:ConseqMainThms} in the introduction.

\par

\renewcommand{\rubrik}{Proposition \ref{Prop:ConseqMainThms}$'$}

\par

\begin{tom}
Let $\varrho \in \mathbf R$, $c\in \mathbf C$,
$\omega ,\omega _\varrho 
\in \mascP _{\! A}(\rr {2d})$ and $p,q\in (0,\infty]$
be such
that $q\le p$ and \eqref{Eq:FracFourWeightsRel} holds.
Then the following is true:
\begin{enumerate}
\item the map
\begin{alignat*}{2}
\mascF _{\! \varrho} =e^{-i\frac \pi 4 H_{x,\varrho ,c}}
\, &:\,
M^{p,q}_{(\omega )}(\rr d)+W^{q,p}_{(\omega )}(\rr d)
& &\to
M^{q,p}_{(\omega _\varrho )}(\rr d)
+
W^{p,q}_{(\omega _\varrho )}(\rr d)
\end{alignat*}
is continuous;

\vrum

\item if in addition $\varrho \notin \mathbf Z$,
then the map
\begin{alignat*}{2}
\mascF _{\! \varrho} =e^{-i\pi H_{x,\varrho ,c}}
\, &:\,
M^{p,q}_{(\omega )}(\rr d)
+
W^{q,p}_{(\omega )}(\rr d) &
&\to
M^{q,p}_{(\omega _\varrho )}(\rr d){\textstyle \bigcap}
W^{p,q}_{(\omega _\varrho )}(\rr d)
\end{alignat*}
is continuous.
\end{enumerate}
\end{tom}

\par

By choosing $p=q$ in (1) in previous proposition, we get
Propositions \ref{Prop:FracFourMod} and 
\ref{Prop:HarmPropMod}.

\medspace

\subsection{Extensions to multiple ordered fractional
Fourier transforms}

\par



By using similar arguments as in the 
proofs of Theorems
\ref{Thm:FrFTModSpec1} and \ref{Thm:FrFTModSpec2},
it follows that
that the following extensions hold true.
Again we refer to Subsections \ref{subsec1.4}
and \ref{subsec1.7} for notations.
The details are left for the reader.

\par

\renewcommand{\rubrik}{Theorem \ref{Thm:FrFTModSpec1}$'$}

\par

\begin{tom}
Let $\varrho \in \rr d\setminus 2\zz d$,
$c\in \mathbf C$,
$\omega ,\omega _\varrho \in \mascP _{\! A}(\rr {2d})$
and $p,q\in (0,\infty]$ be such that $q\le p$ and
\eqref{Eq:FracFourWeightsRel} hold. Then
$\mascF _{\! \varrho}$ and $e^{-i\frac \pi 4H_{x,\varrho ,c}}$
on $\maclH _{\flat _1}'(\rr d)$
restrict to continuous mappings from
$M^{p,q}_{(\omega )}(\rr d)$
to $M^{q,p}_{(\omega _\varrho )}(\rr d)$, and from
$W^{q,p}_{(\omega )}(\rr d)$
to $W^{p,q}_{(\omega _t)}(\rr \varrho )$, and
\begin{alignat}{2}
\nm {\mascF _{\! \varrho} f}{M^{q,p}_{(\omega _\varrho )}}
&=
e^{-\frac \pi 4 \cdot \IM (c)}
\nm {e^{-i\frac \pi 4H_{x,\varrho ,c}}f}
{M^{q,p}_{(\omega _\varrho )}}
\notag
\\
&\lesssim
\big ({\textstyle{\prod _{j=1}^d}}
|\sin (\textstyle{\frac {\pi \varrho _j}2})|
^{\frac 1p-\frac 1q}\big )
\nm f{M^{p,q}_{(\omega )}}, &
\qquad f&\in M^{p,q}_{(\omega )}(\rr d)
\tag*{(\ref{Eq:FrFTModSpec1A})$'$}
\intertext{and}
\nm {\mascF _{\! \varrho} f}{W^{p,q}_{(\omega _\varrho )}}
&=
e^{-\frac \pi 4 \cdot \IM (c)}
\nm {e^{-i\frac \pi 4H_{x,\varrho ,c}}f}
{W^{p,q}_{(\omega _\varrho )}}
\notag
\\
&\lesssim
\big ({\textstyle{\prod _{j=1}^d}}
|\sin (\textstyle{\frac {\pi \varrho _j}2})|
^{\frac 1p-\frac 1q}\big )
\nm f{W^{q,p}_{(\omega )}}, &
\qquad f&\in W^{q,p}_{(\omega )}(\rr d) .
\tag*{(\ref{Eq:FrFTModSpec1B})$'$}
\end{alignat}
\end{tom}

\par

\renewcommand{\rubrik}{Theorem \ref{Thm:FrFTModSpec2}$'$}

\par

\begin{tom}
Let
$\varrho \in \rr d\setminus (2\mathbf Z d+1)^d$,
$\omega ,\omega _\varrho \in \mascP _{\! A}(\rr {2d})$
and $p,q\in (0,\infty]$ be such that $q\le p$ and
\eqref{Eq:FracFourWeightsRel} hold.
Then $\mascF _{\! \varrho} $ and
$e^{-i\frac \pi 4H_{x,\varrho ,c}}$
on $\maclH _{\flat _1}'(\rr d)$
restrict to continuous mappings from
$M^{p,q}_{(\omega )}(\rr d)$
to $W^{p,q}_{(\omega _\varrho )}(\rr d)$, and from
$W^{q,p}_{(\omega )}(\rr d)$
to $M^{q,p}_{(\omega _\varrho )}(\rr d)$, and
\begin{alignat}{2}
\nm {\mascF _{\! \varrho} f}{W^{p,q}_{(\omega _\varrho )}}
&=
e^{-\frac \pi 4 \cdot \IM (c)}
\nm {e^{-i\frac \pi 4H_{x,\varrho ,c}}f}
{W^{p,q}_{(\omega _\varrho )}}
\notag
\\
&\lesssim
\big ({\textstyle{\prod _{j=1}^d}}
|\cos (\textstyle{\frac {\pi \varrho _j}2})|
^{\frac 1p-\frac 1q}\big )
\nm f{M^{p,q}_{(\omega )}}, &
\qquad f&\in M^{p,q}_{(\omega )}(\rr d)
\tag*{(\ref{Eq:FrFTModSpec2A})$'$}
\intertext{and}
\nm {\mascF _{\! \varrho} f}{M^{q,p}_{(\omega _\varrho )}}
&=
e^{-\frac \pi 4 \cdot \IM (c)}
\nm {e^{-i\frac \pi 4H_{x,\varrho ,c}}f}
{M^{q,p}_{(\omega _\varrho )}}
\notag
\\
&\lesssim
\big ({\textstyle{\prod _{j=1}^d}}
|\cos (\textstyle{\frac {\pi \varrho _j}2})|
^{\frac 1p-\frac 1q}\big )
\nm f{W^{q,p}_{(\omega )}}, &
\qquad f&\in W^{q,p}_{(\omega )}(\rr d).
\tag*{(\ref{Eq:FrFTModSpec2B})$'$}
\end{alignat}
\end{tom}

\par

\section{Applications to
Strichartz estimates, and some further continuity
properties for certain partial differential
equations}\label{sec4}

\par

In this section we apply results from
the previous sections to extend certain
Strichartz estimates in \cite{CorNic}, with
initial data in suitable Wiener amalgam
spaces. Thereafter we deduce further
continuity properties for a family of
equations involving certain
Schr{\"o}dinger equations and heat
equations.

\par

\subsection{Strichartz estimates for certain
Schr{\"o}dinger equations}

\par

We shall deduce Strichartz estimates in the
framework of the operator $E$, $S_1$ and $S_2$
in Subsection \ref{subsec1.8} (see
\eqref{Eq:DefEROp}$'$--\eqref{Eq:DefS2ROp}$'$).
If $T>0$, then it follows by straight-forward estimates
that $S_1$ and $S_2$ are continuous
from $C([0,T];M^1(\rr d))$ to
$L^\infty ([0,T];M^1(\rr d))$.


\par

By Proposition
\ref{Prop:HarmPropMod} it follows that
$E$, $S_1$ and $S_2$ are uniquely defined and
continuous. In the following result we
extend the operator $S_j$ to act between
spaces of the form
$L^r([0,T];W^{p,q}(\rr d))$.

\par

\begin{thm}\label{Thm:StrichEstMod1}
Let $p,p_0,q\in [1,\infty ]$ and
$r_0\in (0,\infty ]$ be such that
\begin{equation}\label{Eq:StrichEstMod1LebExp}
0\le d\left (\frac 1q-\frac 1p\right )
<1,
\quad
d\left (\frac 1q-\frac 1p\right )
\le 1+\frac 1{r_0}-\frac 1{p_0},
\end{equation}
with strict inequalities when
$q<p$ and $p_0=1$, or when
$q<p$ and $r_0=\infty$.
Also let $S_1$ and $S_2$
from $C([0,T];M^1(\rr d))$ to
$L^\infty ([0,T];M^1(\rr d))$ be
given by \eqref{Eq:DefS1ROp}$'$ and
\eqref{Eq:DefS2ROp}$'$, and let
$\omega \in \mascP _{\! A ,r}(\rr {2d})$.
Then $S_1$ and $S_2$ is uniquely extendable to a
continuous mappings
\begin{alignat}{2}
S_j &:\, & L^{p_0}([0,T];M^{p,q}_{(\omega )}(\rr d))
&\to
L^{r_0}([0,T];M^{q,p}_{(\omega )}(\rr d)),
\label{Eq:StrichEstMod1}
\\[1ex]
S_j &:\, & L^{p_0}([0,T];M^{p,q}_{(\omega )}(\rr d))
&\to
L^{r_0}([0,T];W^{p,q}_{(\omega )}(\rr d)),
\label{Eq:StrichEstMod2}
\\[1ex]
S_j &:\, & L^{p_0}([0,T];W^{q,p}_{(\omega )}(\rr d))
&\to
L^{r_0}([0,T];M^{q,p}_{(\omega )}(\rr d)),
\label{Eq:StrichEstMod3}
\intertext{and}
S_j &:\, & L^{p_0}([0,T];W^{q,p}_{(\omega )}(\rr d))
&\to
L^{r_0}([0,T];W^{p,q}_{(\omega )}(\rr d)),
\label{Eq:StrichEstMod4}
\end{alignat}
$j=1,2$.
\end{thm}

\par

Theorem \ref{Thm:StrichEstMod1} can also be formulated
in the following way (cf. Theorem \ref{Thm:StrichEstMod1Intro}
in the introduction).

\par

\renewcommand{\rubrik}{Theorem \ref{Thm:StrichEstMod1Intro}$'$}

\par

\begin{tom}
Let $p,p_0,q\in [1,\infty ]$ and
$r_0\in (0,\infty ]$ be such that
\begin{equation*}
0\le d\left (\frac 1q-\frac 1p\right )
<1,
\quad
d\left (\frac 1q-\frac 1p\right )
\le 1+\frac 1{r_0}-\frac 1{p_0},
\end{equation*}
with strict inequalities when
$q<p$ and $p_0=1$, or when
$q<p$ and $r_0=\infty$.
Also let $\omega \in \mascP _{\! A,r}(\rr {2d})$.
Then $S_1$ and $S_2$ from $C([0,T];M^1(\rr d))$ to
$L^\infty ([0,T];M^1(\rr d))$
are uniquely extendable to
continuous mappings
\begin{alignat*}{2}
S_j &:\, & L^{p_0}([0,T];M^{p,q}_{(\omega )}(\rr d)+W^{q,p}_{(\omega )}(\rr d))
&\to
L^{r_0}([0,T];M^{q,p}_{(\omega )}(\rr d)\bigcap W^{p,q}_{(\omega )}(\rr d)),
\end{alignat*}
$j=1,2$, and
\begin{multline}
\nm {S_jF}{L^{r_0}([0,T];M^{q,p}_{(\omega )}(\rr d))}
+
\nm {S_jF}{L^{r_0}([0,T];W^{p,q}_{(\omega )}(\rr d))}
\\[1ex]
\lesssim
\min \left (
\nm {F}{L^{p_0}([0,T];M^{p,q}_{(\omega )}(\rr d))},
\nm {F}{L^{p_0}([0,T];W^{q,p}_{(\omega )}(\rr d))}
\right ) ,\quad j=1,2.
\end{multline}
\end{tom}

\par

We need some preparations for the proof of Theorem
\ref{Thm:StrichEstMod1} and start with the following.

\par

\begin{lemma}\label{Lemma:HardyLittlewoodSobolev}
Suppose that $p_0,q_0,r_0 \in [1,\infty ]$
satisfy
\begin{equation}\label{Eq:YoungIneqCond}
\frac 1{p_0}+\frac 1{q_0}\le
1+\frac 1{r_0}
\quad \text{and}\quad
q_0>1,
\end{equation}
with strict inequality when
$q_0<\infty$ and $p_0=1$, or when $q_0<\infty$ and
$r_0=\infty$. Let $T>0$ and
$\phi _{q_0}(t)=|\sin t|^{-\frac 1{q_0}}$
or $\phi _{q_0}(t)=|\cos t|^{-\frac 1{q_0}}$,
$t\in \mathbf R$. Also let $T_1$ and
$T_2$ be the mappings from $C[0,T]$ to
$C[0,T]$, given by
$$
(T_1h)(t) \equiv \int _0^T \phi _{q_0}(t-s)
h(s)\, ds
\quad \text{and}\quad
(T_2h)(t) \equiv \int _0^t \phi _{q_0}(t-s)
h(s)\, ds,
$$
when $0\le t\le T$. Then $T_1$ and $T_2$ are uniquely extendable to
continuous mappings from $L^{p_0}[0,T]$ to
$L^{r_0}[0,T]$.
\end{lemma}

\par

\begin{proof}
We only prove the assertions for $T_1$.
The assertions for $T_2$ follows by similar
arguments and is left for the reader.

\par

First suppose that $1<p_0$ and $r_0<\infty$,
and let $\psi$ be a measurable
complex-valued function on $\mathbf R$
such that
$$
|\psi (t)|\lesssim |t|^{-\frac 1{q_0}}.
$$
We observe that if $k$ is an integer, then
\begin{alignat*}{3}
|\phi _{q_0}(t)| &\lesssim
|t-k\pi |^{-\frac 1{q_0}} &
\quad &\text{when} &\quad
|t-k\pi |&\le \frac \pi 2,
\intertext{or}
|\phi _{q_0}(t)| &\lesssim
|t-(k+{\textstyle{\frac 12}})
\pi |^{-\frac 1{q_0}} &
\quad &\text{when} &\quad
|t-(k+{\textstyle{\frac 12}})\pi |
&\le \frac \pi 2.
\end{alignat*}
Since $L^p[0,T]$ decreases with $p$, we may
assume that equality is attained in the first
inequality in \eqref{Eq:YoungIneqCond}.
Then it follows from Lebesgue's theorem and
Hardy-Littlewood-Sobolev inequality in
\cite[Theorem 4.5.3]{Ho1} that
$f\mapsto \psi *f$ from $C_0(\mathbf R)$ to
$C(\mathbf R)$ is uniquely extendable to a
continuous map from $L^{p_0}(\mathbf R)$
to $L^{r_0}(\mathbf R)$.

\par

We may now divide $T_1$ into a finite sum
$$
T_1 = \sum _{j=1}^NT_{1,N},
$$
where
$$
(T_{1,j}f)(t) = \int _{a_j}^{b_j}f(t-s)\psi _j(s)\, ds ,
$$
with $\psi _j$ being measurable functions on $\mathbf R$
which satisfy
$$
|\psi _j(t)|\lesssim |t-a_j|^{-\frac 1{q_0}}
\quad \text{or}\quad
|\psi _j(t)|\lesssim |t-b_j|^{-\frac 1{q_0}},
$$
for some $a_j\le b_j$, $j=1,\dots ,N$.
Since each $T_{1,j}$ is uniquely extendable to
a continuous map from $L^{p_0}(\mathbf R)$
to $L^{r_0}(\mathbf R)$, in view of the previous
part of the proof, the asserted continuity 
assertions for $T_1$ follows in the case
$1<p_0$ and $r_0<\infty$.

\par

Next suppose that $p_0=1$.
Then $r_0<q_0$ and $q_0<\infty$, or $r_0=q_0=\infty$.
This implies that
$\phi _{q_0}\in L^{r_0}[0,T]$. By Minkowski's inequality
we obtain
$$
\nm {T_1h}{L^{r_0}[0,T]}
\le
\int _0^T \nm {\phi _{q_0}(\cdo -s)}{L^{r_0}[0,T]}|h(s)|\, ds
\le 2\nm {\phi _{q_0}}{L^{r_0}[0,T]}\nm h{L^1[0,T]},
$$
and the result follows in the case $p_0=1$.

\par

Finally suppose that $r_0=\infty$ and $q_0<\infty$. Then
$p_0'<q_0$, giving that $\phi _{q_0}\in L^{p_0'}[0,T]$.
Hence, H{\"o}lder's inequality gives
$$
|(T_1h)(t)|
\le
\nm {\phi _{q_0}(t-\cdo )}{L^{p_0'}[0,T]}\nm h{L^{p_0}[0,T]}
\le
2\nm {\phi _{q_0}}{L^{p_0'}[0,T]}\nm h{L^{p_0}[0,T]},
$$
and the result follows.
\end{proof}

\par

\begin{proof}[Proof of Theorem \ref{Thm:StrichEstMod1}]
We only prove the continuity for 
\eqref{Eq:StrichEstMod1}.
The other cases follow by similar arguments
and are left for the reader.

\par

Since $L^{r_0}[0,T]$ is decreasing with $r_0$
and that \eqref{Eq:StrichEstMod1LebExp}
is obviously true for some $r_0\ge 1$, we may
assume that $r_0\ge 1$. We also observe that
\eqref{Eq:StrichEstMod1LebExp} implies that
$q\le p$, which makes it possible to apply
Theorem \ref{Thm:FrFTModSpec1}.

\par

Let $q_0\in \mathbf R\cup \{ \infty \}$ be defined by
$$
\frac 1{q_0} = d\left (\frac 1q-\frac 1p\right )
$$
and let $\phi _{q_0}$ be the same as in Lemma
\ref{Lemma:HardyLittlewoodSobolev}.
Then $q_0\in (1,\infty ]$. A combination of
Theorem \ref{Thm:FrFTModSpec1} and Minkowski's
inequality gives
\begin{align*}
\nm {(SF)(t,\cdo )}{M^{q,p}_{(\omega )}}
&\le
\int _0^T
\nm {(e^{-i\pi (t-s)H_{x,c}}F(s,\cdo ))}{M^{q,p}_{(\omega )}}
\, ds
\\[1ex]
&\lesssim
\int _0^T
\phi _{q_0}(t-s)\nm {F(s,\cdo ))}{M^{p,q}_{(\omega )}}
\, ds .
\end{align*}
By applying the $L^{r_0}[0,T]$ norm on the last
inequality and using Lemma
\ref{Lemma:HardyLittlewoodSobolev} we get
$$
\nm {SF}{L^{r_0}([0,T];M^{q,p}_{(\omega )})}
\lesssim
\nm {SF}{L^{p_0}([0,T];M^{p,q}_{(\omega )})},
$$
and the result follows.
\end{proof}

\par

Next we shall apply Theorems \ref{Thm:FrFTModSpec1}
and \ref{Thm:StrichEstMod1} to deduce continuity
properties for the operator $E$ in
\eqref{Eq:DefERAdjOp}$'$, and thereby obtain
Strichartz estimates for the harmonic oscillator
propagator $e^{-itH_{x,c}}$. The first
result is the following and is a straight-forward
consequence of Theorem \ref{Thm:FrFTModSpec1}.
The details are left for the reader.
Here we let $L^p_{\text w}(\Omega )$ be the weak
$L^p(\Omega )$ space when $p\in (0,\infty ]$,
which is often denoted by
$L^{p,\infty}(\Omega )$ or $L^{p,*}(\Omega )$
in the literature (see e.{\,}g. \cite{BeLo}).

\par

\begin{thm}\label{Thm:StrichEstMod2}
Let 
$c\in \mathbf C$,
$\omega\in \mascP _E^r(\rr {2d})$ and
$p,q,r_0\in (0,\infty]$ be such that $q\le p$
and
\begin{equation}\label{Eq:StrichEstMod2LebExpCond}
\frac 1{r_0}=d\left (\frac 1q -\frac 1p\right ).
\end{equation}
Then $E$ in \eqref{Eq:DefERAdjOp}$'$
is uniquely extendable to a continuous map
$$
E : M^{p,q}_{(\omega )}(\rr d)
+
W^{q,p}_{(\omega )}(\rr d)
\to
L^{r_0}_{{\text {\rm {w}}}}([0,T];M^{q,p}_{(\omega )}(\rr d))
\bigcap
L^{r_0}_{{\text {\rm {w}}}}([0,T];W^{p,q}_{(\omega )}(\rr d)),
$$
and
\begin{multline}
\nm {Ef}{L^{r_0}_{{\text {\rm {w}}}}
([0,T];M^{q,p}_{(\omega )})}
+
\nm {Ef}{L^{r_0}_{{\text {\rm {w}}}}
([0,T];W^{p,q}_{(\omega )})}
\lesssim
\nm f{ M^{p,q}_{(\omega )}
+
W^{q,p}_{(\omega )}},
\\[1ex]
f\in M^{p,q}_{(\omega )}(\rr d)
+
W^{q,p}_{(\omega )}(\rr d).
\label{Eq:StrichEstMod2A}
\end{multline}
\end{thm}

\par

Here recall that if $\maclB _1$ and $\maclB_2$
are quasi-Banach spaces, then $\maclB _1+\maclB _2$
is the quasi-Banach space
$$
\sets{f_1+f_2}{f_1\in \maclB _1,\ f_2\in \maclB _2}
$$
equipped with the quasi-norm
$$
\nm f{\maclB_1+\maclB _2}
\equiv
\inf _{f=f_1+f_2} \left (
\nm {f_1}{\maclB _1}+
\nm {f_2}{\maclB _2}
\right ).
$$

\par

\begin{rem}
If $p=q$ in Theorem \ref{Thm:StrichEstMod2}, then
$r_0=\infty$ in \eqref{Eq:StrichEstMod2LebExpCond}.
By Theorems \ref{Thm:FrFTModSpec1}
and \ref{Thm:FrFTModSpec2} it follows that $E$ in
\eqref{Eq:DefERAdjOp}$'$
is uniquely extendable to a continuous map
$$
E : M^{p}_{(\omega )}(\rr d)
\to
L^{\infty}(\mathbf R;M^{p}_{(\omega )}(\rr d)),
$$
and
\begin{equation}
\tag*{(\ref{Eq:StrichEstMod2A})$'$}
\nm {Ef}{L^{\infty}(\mathbf R;M^{p}_{(\omega )})}
\lesssim
\nm f{M^{p}_{(\omega )}},
\qquad
f\in M^{p}_{(\omega )}(\rr d).
\end{equation}
\end{rem}

\par

The next result follows by combining Theorem
\ref{Thm:StrichEstMod1} with general techniques
in Section 2 in \cite{GinVel} for Strichartz estimates
in order to deduce further continuity properties for
$E$ and $E^*$.

\par

\begin{thm}\label{Thm:StrichEstMod3}
Let $p,p_0\in [1,\infty ]$ be such that
\begin{equation}\label{Eq:StrichEstMod1LebExp2}
2\le p<\frac {2d}{d-1}
\quad \text{and}\quad
d\left ( 1-\frac 2p\right ) = \frac 2{p_0'}.
\end{equation}
Then $E^*$ in \eqref{Eq:DefERAdjOp}$'$
is uniquely extendable to
continuous mappings
\begin{alignat}{2}
E^* &:\, & L^{p_0}([0,T];M^{p,p'}(\rr d))
&\to
L^2(\rr d),
\label{Eq:StrichEstMod21}
\\[1ex]
E^* &:\, & L^{p_0}([0,T];W^{p',p}(\rr d))
&\to
L^2(\rr d),
\label{Eq:StrichEstMod22}
\end{alignat}
and  $E$ in \eqref{Eq:DefEROp}$'$
is uniquely extendable to
continuous mappings
\begin{alignat}{2}
E &:\, & L^2(\rr d)
&\to
L^{p_0'}([0,T];M^{p',p}(\rr d)),
\label{Eq:StrichEstMod23}
\\[1ex]
E &:\, & L^2(\rr d)
&\to
L^{p_0'}([0,T];W^{p,p'}(\rr d)),
\label{Eq:StrichEstMod24}
\end{alignat}
\end{thm}

\par

\begin{proof}
By choosing $r_0=p_0'$, $q=p'$ and $\omega =1$,
it follows that \eqref{Eq:StrichEstMod1} and
\eqref{Eq:StrichEstMod4} takes the forms
\begin{alignat}{2}
S_2 &:\, & L^{p_0}([0,T];M^{p,p'}(\rr d))
&\to
L^{p_0'}([0,T];M^{p',p}(\rr d))
\tag*{(\ref{Eq:StrichEstMod1})$'$}
\intertext{and}
S_2 &:\, & L^{p_0}([0,T];W^{p',p}(\rr d))
&\to
L^{p_0'}([0,T];W^{p,p'}(\rr d)).
\tag*{(\ref{Eq:StrichEstMod4})$'$}
\end{alignat}
Since the ranks in \eqref{Eq:StrichEstMod1}$'$
and \eqref{Eq:StrichEstMod4}$'$ are the duals to
their domains, the result follows by straight-forward
applications of the equivalences (2.1)--(2.3) in
\cite{GinVel}.
\end{proof}

\par

By Theorem \ref{Thm:StrichEstMod3} and the
fact that $S_2=E\circ E^*$, we get the following
(see also \cite[Corollary 2.1]{GinVel}).

\par

\begin{thm}\label{Thm:StrichEstMod4}
Let $p_j,p_{0,j}\in [1,\infty ]$
be such that
\begin{equation}\label{Eq:StrichEstMod1LebExp4}
2\le p_j<\frac {2d}{d-1}
\quad \text{and}\quad
d\left ( 1-\frac 2{p_j}\right ) = \frac 2{p_{0,j}'},
\end{equation}
$j=1,2$. Then $S_2$ in \eqref{Eq:DefS2ROp}$'$
is uniquely extendable to
continuous mappings
\begin{alignat}{2}
S_2 &:\, & L^{p_{0,1}}([0,T];M^{p_1,p_1'}(\rr d))
&\to
L^{p_{0,2}'}([0,T];M^{p_2',p_2}(\rr d)),
\label{Eq:StrichEstMod41}
\\[1ex]
S_2 &:\, & L^{p_{0,1}}([0,T];M^{p_1,p_1'}(\rr d))
&\to
L^{p_{0,2}'}([0,T];W^{p_2,p_2'}(\rr d)),
\label{Eq:StrichEstMod42}
\\[1ex]
S_2 &:\, & L^{p_{0,1}}([0,T];W^{p_1',p_1}(\rr d))
&\to
L^{p_{0,2}'}([0,T];M^{p_2',p_2}(\rr d)),
\label{Eq:StrichEstMod43}
\intertext{and}
S_2 &:\, & L^{p_{0,1}}([0,T];W^{p_1',p_1}(\rr d))
&\to
L^{p_{0,2}'}([0,T];W^{p_2,p_2'}(\rr d)).
\label{Eq:StrichEstMod44}
\end{alignat}
The same holds true with $S_1$ in
place of $S$ at each occurrence.
\end{thm}

\par

\begin{rem}
We observe that \eqref{Eq:StrichEstMod44} is the
same as (45) in \cite{CorNic}.
\end{rem}

\par

\begin{rem}
We observe that the conditions \eqref{Eq:StrichEstMod1LebExp2}
and \eqref{Eq:StrichEstMod1LebExp4} imply that $p_0<2$
and $p_{0,j}<2$.
\end{rem}

\par

\subsection{Continuity properties for a
family of equations related to Schr{\"o}dinger
and heat equations}

\par

Next we consider more general equations, given by
\eqref{Eq:GenSchrEqParPot}, where $u_0$
is a suitable function or ultra-distribution on $\rr d$,
$F$ is a suitable function or ultra-distribution on
$\rr {d+1}$ and $R=H_{x,\varrho ,c}^r$ for some
$r\ge 0$, $\varrho \in \cc d$ and $c \in \mathbf C$.
That is, we consider
\begin{equation}\label{Eq:GenSchrEqParPotThird}
\begin{cases}
i\partial _tu-H_{x,\varrho ,c}^ru=F, 
\\[1ex]
u(0,x)=u_0(x),\qquad
(t,x)\in [0,T]\times \rr d,
\end{cases}
\end{equation}
where $T>0$ is fixed. Here we observe that
\eqref{Eq:GenSchrEqParPotThird} is a partial
differential equation when $r$ is an integer.
For general $r$, \eqref{Eq:GenSchrEqParPotThird}
becomes a pseudo-differential equation.

\par

By \eqref{Eq:FormalSol} it follows that the formal
solution is given by
\begin{equation}\label{Eq:FormalSolSec}
u(t,x) = (e^{-itH_{x,\varrho ,c}^r}u_0)(x)
-i\int _0^t (e^{i(t-s)H_{x,\varrho ,c}^r}
F(t,\cdo ))(x)\, ds.
\end{equation}
Hence, questions on well-posed properties for
the equation \eqref{Eq:GenSchrEqParPotThird}
rely completely on continuity properties of the
propagator
\begin{alignat}{2}
(E_{\varrho ,r,c}f)(t,x)
&\equiv
(e^{-itH_{x,\varrho ,c}^r}u_0)(x), 
\qquad
(t,x) \in [0,T]\times \rr d.
\label{Eq:DefEROpSec}
\intertext{as well as for the operators}
(S_{1,\varrho ,r,c}F)(t,x)
&= 
\int _0^t (e^{-i(t-s)H_{x,\varrho ,c}^r}F(s,\cdo ))(x)
\, ds, 
\qquad
(t,x) \in [0,T]\times \rr d,
\label{Eq:DefS1ROpSec}
\intertext{and}
(S_{2,\varrho ,r,c}F)(t,x)
&= 
\int _0^T (e^{-i(t-s)H_{x,\varrho ,c}^r}F(s,\cdo ))(x)
\, ds, 
\qquad
(t,x) \in [0,T]\times \rr d.
\label{Eq:DefS2ROpSec}
\end{alignat}
Here we remark that there are different definitions of
well-posed problems in the literature.
We say that the problem \eqref{Eq:GenSchrEqParPotThird}
is \emph{well-posed} if the solution $u=u(t,x)$ depends
continuously on the initial data $u_0$. If
\eqref{Eq:GenSchrEqParPotThird} fails to be well-posed, then 
\eqref{Eq:GenSchrEqParPotThird} is called \emph{ill-posed}.

\par

By Proposition \ref{Prop:BargPilSpacesNonCont}
it follows that the following is true, which shows that
the operator \eqref{Eq:DefEROpSec} easily become
discontinuous in the framework of classical
functions and (ultra-)distribution spaces. The details are
left for the reader.

\par

\begin{prop}\label{Prop:DiscontPropGSSpaces}
Let $r,T>0$, $c\in \mathbf C$ and $\varrho \in \cc d$
be such that $\operatorname{Im}(\varrho _j^r)>0$
for some $j\in \{1,\dots ,d\}$. Then following is true:
\begin{enumerate}
\item $E_{\varrho ,r,c}$ in \eqref{Eq:DefEROpSec}
is discontinuous from $\mascS (\rr d)$
to $\mascS '(\rr d)$;

\vrum

\item if in addition $r\ge 1$, then
$E_{\varrho ,r,c}$ in \eqref{Eq:DefEROpSec}
is discontinuous from $\maclS _{1/2}(\rr d)$
to $\maclS _{1/2} '(\rr d)$.
\end{enumerate}
\end{prop}

\par

On the other hand, by Proposition
\ref{Prop:BargPilSpaces}, it follows that
the mappings 
\eqref{Eq:DefEROpSec}, \eqref{Eq:DefS1ROpSec}
and \eqref{Eq:DefS2ROpSec}
are continuous on suitable
Pilipovi{\'c} spaces, which is explained in
the following result. The details are left for the reader.
Here we let
\begin{align*}
L^1([0,T];\maclH _{0,s}(\rr d))
&=
\underset{r>0}\projlim
L^1([0,T];\maclH _{s;r}(\rr d)),
\\[1ex]
L^1([0,T];\maclH _s(\rr d))
&=
\underset{r>0}\indlim
L^1([0,T];\maclH _{s;r}(\rr d)),
\\[1ex]
L^1([0,T];\maclH _s'(\rr d))
&=
\underset{r>0}\projlim
L^1([0,T];\maclH _{s;r}'(\rr d)),
\intertext{and}
L^1([0,T];\maclH _{0,s}'(\rr d))
&=
\underset{r>0}\indlim
L^1([0,T];\maclH _{s;r}'(\rr d)),
\end{align*}
where $\maclH _{s;r}(\rr d)$ and $\maclH _{s;r}'(\rr d)$
are the images of $\ell _{s;r}^\infty (\nn d)$ and
$\ell _{s;r}^{\infty ,*} (\nn d)$, respectively under the map
$T_{\maclH}$ in \eqref{Eq:THMap}, also in topological
sense. (See also Definition \ref{Def:SeqSpaces}.)

\par

\begin{thm}\label{Thm:ContPropPilSpaces}
Let $r,T>0$, $c\in \mathbf C$, $\varrho \in \cc d$
and let $s,s_1,s_2\in \overline {\mathbf R_{\flat}}$
be such that $0<s_1\le \frac 1{2r}$ and
$s_2< \frac 1{2r}$. Then following is true:
\begin{enumerate}
\item $E_{\varrho ,r,c}$ in \eqref{Eq:DefEROpSec}
is a homeomorphism on the spaces in
\eqref{Eq:Pilspaces2};

\vrum

\item $S_{j,\varrho ,r,c}$ in \eqref{Eq:DefS1ROpSec}
and \eqref{Eq:DefS2ROpSec} are homeomorphisms
on the spaces
\begin{equation}\label{Eq:Pilspaces2Strich}
\begin{alignedat}{2}
&L^1([0,T];\maclH _{0,s_1}(\rr d)), &
\quad
&L^1([0,T];\maclH _{s_2}(\rr d)),
\\[1ex]
&L^1([0,T];\maclH _{s_2}'(\rr d)) &
\quad \text{and}\quad
&L^1([0,T];\maclH _{0,s_1}'(\rr d)).
\end{alignedat}
\end{equation}
\end{enumerate}
\end{thm}

\par

As consequences of Proposition 
\ref{Prop:DiscontPropGSSpaces}
and Theorem \ref{Thm:ContPropPilSpaces}
we get the following, concerning well-posed
properties of the equation
\eqref{Eq:GenSchrEqParPotThird}.

\par

\begin{cor}\label{Cor:ContPropPilSpaces}
Suppose that $r\ge 1$,
$c\in \mathbf C$ and $\varrho \in \cc d$
are such that $\operatorname{Im}(\varrho _j^r)>0$
for some $j\in \{1,\dots ,d\}$.
Then the following is true:
\begin{enumerate}
\item the equation \eqref{Eq:GenSchrEqParPotThird}
is ill-posed in the framework of
Schwartz functions, Gelfand-Shilov spaces,
and their dual spaces;

\vrum

\item the equation \eqref{Eq:GenSchrEqParPotThird}
is well-posed for the Pilipovi{\'c} spaces
and their dual spaces in \eqref{Eq:Pilspaces2}.
\end{enumerate}
\end{cor}

\par

\begin{rem}
If $c=0$, $r=1$, $F=0$ and $\varrho _j=i$ for every $j$ in
Corollary \ref{Cor:ContPropPilSpaces}, then
\eqref{Eq:GenSchrEqParPotThird} takes the form
\begin{equation}\tag*{(\ref{Eq:GenSchrEqParPotThird})$'$}
\begin{cases}
\partial _tu=-\Delta _x+|x|^2, 
\\[1ex]
u(0,x)=u_0(x),\qquad
(t,x)\in [0,T]\times \rr d.
\end{cases}
\end{equation}
That is we obtain a sort of heat equation, where the potential
term $|x|^2$ is included. The minus sign in front of the Laplace
operator $\Delta _x$ implies that we are searching for a solution
when moving backwards in time.

\par

Corollary \ref{Cor:ContPropPilSpaces} then shows that it is
\emph{not meaningful} to investigate
\eqref{Eq:GenSchrEqParPotThird} in
the framework of Schwartz spaces, Gelfand-Shilov spaces
and their distribution spaces. On the other hand, it
follows from the same corollary that it always makes sense
to investigate such problems in the framework of Pilipovi{\'c}
spaces which are not Gelfand-Shilov spaces, and their
distributions.
\end{rem}


\par

\subsection{Some further applications and remarks}

\par

A question which appears is whether our results are applicable
to problems like
\begin{equation}\label{Eq:GenSchrEqParPotFourth}
\begin{cases}
\partial _tu-\op ^{w}(a)u=F, 
\\[1ex]
u(0,x)=u_0(x),\qquad
(t,x)\in [0,T]\times \rr d,
\end{cases}
\end{equation}
when $a$ is a positive definite quadratic form on $\rr {2d}$.
Here $\op ^{w}(a)$ is the \emph{Weyl quantization} of
$a$, i.{\,}e. the operator on $\mascS (\rr d)$, given by
$$
\op ^w(a)f(x)=(2\pi )^{-d}\iint _{\rr {2d}}
a({\textstyle{\frac 12}}(x+y),\xi )f(y)e^{i\scal {x-y}\xi}\, dyd\xi .
$$
The operator $\op ^w(a)$ is continuous on $\mascS (\rr d)$
and on any Pilipovi{\'c} space,
which extends uniquely to a continuous map on $\mascS '(\rr d)$,
and to any Pilipovi{\'c} distribution space.
(See e.{\,}g. \cite{Ho1,Toft26}.)

\par

By introducing suitable new symplectic coordinates it
follows that $\op ^w(a)$ takes the form
$$
\op ^w(a) = \sum _{j=1}^d\varrho _j(x_j^2-\partial _{x_j}^2)+c,
$$
for some $\varrho _j>0$, $j=1,\dots ,d$, and some real
constant $c$, in these new coordinates. (See Section 18.6
in \cite{Ho1}.)

\par

By Proposition
\ref{Prop:DiscontPropGSSpaces} and Theorem
\ref{Thm:ContPropPilSpaces} it follows that
\eqref{Eq:GenSchrEqParPotFourth} is ill-posed
in the framework of Schwartz spaces, Gelfand-Shilov
spaces and their distribution spaces, but
well-posed for other types of Pilipovi{\'c} spaces of functions
and distributions. Here we notice that we need to keep
staying in these new symplectic coordinates, since Pilipovi{\'c}
spaces which are not Gelfand-Shilov spaces, are not
invariant under general changes of symplectic coordinates.
(See e.{\,}g. \cite{Toft18}.)

\par

\end{document}